\newif\ifpictures
\picturestrue

\documentclass[12pt]{amsart}
\usepackage{amssymb,amsmath}
 \usepackage{amsopn}
 \usepackage{xspace}
  \usepackage{hyperref}
 \usepackage[dvips]{graphicx}
\usepackage[arrow,matrix,curve]{xy}
\usepackage {color, tikz}

\numberwithin{equation}{section}
\newtheorem{thm}{Theorem}
\newtheorem{example}[thm]{Example}
\newtheorem{prop}[thm]{Proposition}
\newtheorem{lemma}[thm]{Lemma}
\newtheorem{cor}[thm]{Corollary}

\newtheorem{defn}[thm]{Definition}
\newtheorem{infthm}[thm]{Informal Statement}

\numberwithin{thm}{section}

\newcounter{FNC}[page]
\def\newfootnote#1{{\addtocounter{FNC}{2}$^\fnsymbol{FNC}$%
     \let\thefootnote\relax\footnotetext{$^\fnsymbol{FNC}$#1}}}

\newcommand{\C}{\mathbb{C}}
\newcommand{\N}{\mathbb{N}}
\newcommand{\Q}{\mathbb{Q}}
\newcommand{\R}{\mathbb{R}}

\newcommand{\F}{\mathbb{F}}

\newcommand{\Z}{\mathbb{Z}}

\newcommand{\lf}{\left}
\newcommand{\ri}{\right}
\newcommand{\ra}{\rightarrow}
 
\newcommand{\Ra}{\Rightarrow}

\newcommand{\Lera}{\Leftrightarrow}
\newcommand{\ovl}{\overline}
\newcommand{\wh}{\widehat}

\newcommand\cA{{\ensuremath{\mathcal{A}}}\xspace}

\newcommand\cT{{\ensuremath{\mathcal{T}}}\xspace}

\newcommand\cV{{\ensuremath{\mathcal{V}}}\xspace}

\newcommand{\vphi}{\varphi}
\newcommand{\alp}{\alpha}
\newcommand{\lam}{\lambda}

\newcommand{\Sig}{\Sigma}

\DeclareMathOperator{\sgn}{sgn}
\DeclareMathOperator{\conv}{conv}
\DeclareMathOperator{\supp}{supp}

\DeclareMathOperator{\trop}{trop}

\DeclareMathOperator{\Log}{Log}
\DeclareMathOperator{\Exp}{Exp}
\DeclareMathOperator{\New}{New}

\DeclareMathOperator{\RE}{Re}
\DeclareMathOperator{\IM}{Im}

\DeclareMathOperator{\Mat}{Mat}
\DeclareMathOperator{\Int}{int}
\DeclareMathOperator{\SOS}{SOS}
\DeclareMathOperator{\lcm}{lcm}
\DeclareMathOperator{\V}{vert}

\title[Amoebas, Nonnegative Polynomials and SOS Supported on Circuits]{Amoebas, Nonnegative Polynomials and Sums of Squares Supported on Circuits}

\begin{document}

\author{Sadik Iliman}
\address{Sadik Iliman, Goethe-Universit\"at, FB 12 -- Institut f\"ur Mathematik,
Postfach 11 19 32, D-60054 Frankfurt am Main, Germany}
\email{iliman@math.uni-frankfurt.de}

\author{Timo de Wolff}
\address{Timo de Wolff, Texas A\&M University, Department of Mathematics, College Station, TX 77843-3368, 
 USA\medskip}
 \email{dewolff@math.tamu.edu}


\begin{abstract}
We completely characterize sections of the cones of nonnegative polynomials, convex polynomials and sums of squares with polynomials supported on circuits, a genuine class of sparse polynomials. In particular, nonnegativity is characterized by an invariant, which can be immediately derived from the initial polynomial. Furthermore, nonnegativity of such polynomials $f$ coincides with solidness of the amoeba of $f$, i.e., the Log-absolute-value image of the algebraic variety $\cV(f) \subset (\C^*)^n$ of $f$. 

These results generalize earlier works both in amoeba theory and real algebraic geometry by Fidalgo, Kovacec, Reznick, Theobald and de Wolff and solve an open problem by Reznick. They establish the first direct connection between amoeba theory and nonnegativity of real polynomials. Additionally, these statements yield a completely new class of nonnegativity certificates independent from sums of squares certificates. 
\end{abstract}

\maketitle

\section{Introduction}
Forcing additional structure on polynomials often simplifies certain problems in theory and practice. One of the most prominent examples is given by \emph{sparse} polynomials, which arise in different areas in mathematics. Exploiting sparsity in problems can reduce the complexity of solving hard problems. An important example is, given by sparse polynomial optimization problems, see \cite{Lasserre:sparse}. In this paper, we consider sparse polynomials having a special structure in terms of their Newton polytopes and supports. More precisely, we look at polynomials $f\in \R[\mathbf{x}] = \R[x_1,\dots,x_n]$, whose Newton polytopes are simplices and the supports are given by all the vertices of the simplices and one additional interior lattice point in the simplices. Such polynomials have exactly $n + 2$ monomials and can be regarded as \emph{supported on a circuit}. Note that $A \subset \mathbb N^n$ is called a \emph{circuit}, if $A$ is affinely dependent, but any proper subset of $A$ is affinely independent, see \cite{GKZ:discriminant}.  We write these polynomials as
\begin{eqnarray}
	f & = & \sum_{j=0}^n b_j \mathbf{x}^{\alpha(j)} + c \mathbf{x}^y \label{Equ:OurPolynomials}
\end{eqnarray}
where the Newton polytope $\Delta = \New(f) = \conv\{\alpha(0), \dots, \alpha(n)\}\subset \R^n$ is a lattice simplex, $y \in \Int(\Delta)$, $b_j \in \R_{>0}$ and $c \in \R^*$. We denote this class of polynomials as $P_{\Delta}^y$.
In this setting, the goal of this paper is to connect and establish new results in two different areas of mathematics. Namely, we link amoeba theory with nonnegative polynomials and sums of squares. The theory of amoebas deals with images of
varieties $\mathcal \cV(f) \subset (\C^*)^n$ under the Log-absolute-value map
\begin{eqnarray}
	\Log|\cdot|: \ \lf(\C^*\ri)^n \to \R^n, \quad (z_1,\ldots,z_n) \mapsto (\log|z_1|,\ldots, \log|z_n|), \label{Equ:LogMap}
\end{eqnarray}
having their nature in complex algebraic geometry with applications in various mathematical subjects including complex analysis \cite{Forsberg:Passare:Tsikh,GKZ:discriminant}, the topology of real algebraic curves \cite{Mikhalkin:Annals}, dynamical systems \cite{Einsiedler:et:al}, dimers / crystal shapes \cite{Kenyon:Okounkov:Sheffield}, and in particular with strong connections to tropical geometry, see \cite{Mikhalkin:Survey,Passare:Tsikh:Survey}. The cones of nonnegative polynomials and sums of squares arise as central objects in convex algebraic geometry and polynomial optimization, see \cite{Blekherman:Parrilo:Thomas, Lasserre:Buch}. 

For both amoebas and nonnegative polynomials / sums of squares, work has been done for special configurations in the above setting. In \cite{TTdW:genusone}, the authors give a characterization of the corresponding amoebas of such polynomials and in \cite{Fidalgo:Kovacec, Reznick:AGI}, the authors characterize questions of nonnegativity and sums of squares for very special coefficients and simplices in the above sparse setting. We aim to extend results in all of these papers and establish connections between them for polynomials $f\in P_{\Delta}^y$.\\

We call a lattice point $\alp \in \Z^n$ \textit{even} if every entry $\alp_j$ is even, i.e., $\alp \in (2\Z)^n$. We call an integral polytope \textit{even} if all its vertices are even. Finally, we call a polynomial a \textit{sum of monomial squares} if all monomials $b_\alp \mathbf{x}^{\alp}$ satisfy $b_{\alp} > 0$ and $\alp$ even.\\

For the remainder of this article we assume that \textbf{every} polytope is even unless it is explicitly stated otherwise. However, we will reemphasize this fact in key statements.

For $f \in P_{\Delta}^y$ we define the \textit{circuit number} $\Theta_f$ as
\begin{eqnarray}
	\Theta_f & = & \prod_{j = 0}^n \left(\frac{b_j}{\lam_j}\right)^{\lam_j}, \label{Equ:CircuitNumber}
\end{eqnarray}
where the $\lam_j$ are uniquely given by the convex combination $\sum_{j = 0}^n \lam_j \alp(j) = y,\lambda_j \geq 0, \sum_{j = 0}^n \lam_j = 1$. We show that every polynomial $f \in P_{\Delta}^y$ is, up to an isomorphism on $\R^n$, completely characterized by the $\lam_j$ and its circuit number $\Theta_f$. 

Remember that we always have $c \in \R^*$ by definition of $P_{\Delta}^y$. The case $c = 0$ implies that the polynomial $f$ is a sum of monomial squares and hence always is nonnegative. This should be kept in mind when with slight abuse of notation $c = 0$ is a possible choice in some statements. We now formulate our main theorems. The first theorem stated here is a composition of Theorem \ref{Thm:Positiv} and the Corollaries \ref{cor:zerobound}, \ref{Cor:ADiscriminant} and \ref{Cor:AmoebaSolidness} in the article.

\begin{thm}
Let $f \in P_{\Delta}^y$ and $\Delta$ be an even simplex, i.e., $\alp(j) \in (2\N)^n$ for all $0 \leq j \leq n$. Then the following statements are equivalent.
\begin{enumerate}
    \item $c \in [-\Theta_f,\Theta_f]$ for $y \notin (2\mathbb N)^n$ and $c \geq -\Theta_f$ for $y \in (2\N)^n$.
    \item $f$ is nonnegative.
\end{enumerate}

Furthermore, $f$ is located on the boundary of the cone of nonnegative polynomials if and only if $|c| = \Theta_f$ for $y \notin (2\N)^n$ and $c = -\Theta_f$ for $y \in (2\N)^n$. In these cases, $f$ has at most $2^n$ real zeros all of which only differ in their signs.

Assume that furthermore $n \geq 2$ and $f$ is not a sum of monomial squares with $c > 0$. Then the following are equivalent.
\begin{enumerate}
 \item $f$ is nonnegative, i.e., $c \in [-\Theta_f,\Theta_f]$ for $y \notin (2\mathbb N)^n$ and $c \in [-\Theta_f,0]$ for $y \in (2\N)^n$
 \item The amoeba $\cA(f)$ is solid.
\end{enumerate}
\label{Thm:MainEquivalences}
\end{thm}

Note in this context that an amoeba $\cA(f)$ of $f \in P_\Delta^y$ is \textit{solid} if and only if its complement has no bounded components. Note furthermore that since $\Delta$ is an even simplex, $f$ is a sum of monomial squares (and hence trivially nonnegative) if and only if $c \geq 0$ and $y \in (2\N)^n$.

Theorem \ref{Thm:MainEquivalences} yields a very interesting relation between the structure of the amoebas of $f \in P_{\Delta}^y$ and nonnegative polynomials $f \in P_{\Delta}^y$, which are both completely characterized by the circuit number. Furthermore, it generalizes amoeba theoretic results from \cite{TTdW:genusone}.\\

A crucial observation for $f \in P_{\Delta}^y$ is that nonnegativity of such $f$ does not imply that $f$ is a sum of squares. It is particularly interesting that the question whether $f \in P_{\Delta}^y$ is a sum of squares or not depends on the lattice point configuration of the Newton polytope of $f$ alone. We give a precise characterization of the nonnegative $f \in P_{\Delta}^y$ which are additionally a sum of squares in Section \ref{Sec:PSDSOS}, Theorem \ref{thm:sos}. Here, we present a rough version of the statement.

\begin{infthm}
Let $f \in P_{\Delta}^y$ and $\Delta$ be an even simplex. Let $f$ be nonnegative. Then $f$ is a sum of squares if and only if $y$ is the midpoint of two even distinct lattice points contained in a particular subset of lattice points in $\Delta$. In particular, this is independent of the choice of the coefficients $b_j, c$.
\label{Thm:MainSOS}
\end{infthm}

Note that Theorems \ref{Thm:MainEquivalences} and \ref{Thm:MainSOS} generalize the main results in \cite{Fidalgo:Kovacec} and \cite{Reznick:AGI} and yield them as special instances. In Section \ref{Sec:PSDSOS} we will explain this relationship in more detail. 

Based on these characterizations we define a new convex cone $C_{n,2d}$:

\begin{defn}
 We define the set of \emph{sums of nonnegative circuit polynomials} (SONC) as
$$ C_{n,2d} \ = \ \left\{f \in \R[\mathbf{x}]_{2d} \ :\  f = \sum_{i=1}^k \lambda_i g_i, \lambda_i \geq 0, g_i \in P_{\Delta_i}^y\cap P_{n,2d}\right\}$$
for some even lattice simplices $\Delta_i \subset \R^n$.
\label{Def:SONC}
\end{defn}

It follows by construction that membership in the $C_{n,2d}$ cone serves as a nonnegativity certificate, see also Proposition \ref{Prop:ConeContainment}.

\begin{cor}
Let $f \in \R[\mathbf{x}]$. Then $f$ is nonnegative if there exist $\mu_i \geq 0$, $g_i \in C_{n,2d}$ for $1 \leq i\leq k$ such that
$$f \ = \ \sum_{i=1}^k \mu_i g_i.$$
\end{cor}

In Section \ref{Sec:SOPC} we discuss the SONC cone in further detail. In Proposition \ref{Prop:ConeContainment} we show that the SONC cone and the SOS cone are not contained in each other for general $n$ and $d$. Particularly, we also prove that the existence of a SONC decomposition is equivalent to nonnegativity of $f$ if $\New(f)$ is a simplex and there exists an orthant where all terms of $f$ except for  those corresponding to vertices have a negative sign (Corollary \ref{Cor:SONC}).\\

Finally, we prove the following result about convexity, see also Theorem \ref{thm:konvex}.

\begin{thm}
Let $n \geq 2$ and $f \in P_{\Delta}^y$ where $\Delta$ is an even simplex. Then $f$ is not convex.
\label{Thm:MainConvex}
\end{thm}

Recently, there is much interest in understanding the cone of convex polynomials, Theorem \ref{Thm:MainConvex} serves as an indication that sparsity is a structure that can prevent polynomials from being convex.\\

\textbf{Further contributions.} 
$\mathbf{1}$. Gale duality is a standard concept for (convex) polytopes, matroids and sparse polynomial systems, see \cite{Bjoerner:et:al:Matroids,GKZ:discriminant,Gruenbaum,Sottile:Book:RealSolutions}. We show that a polynomial $f \in P_{\Delta}^y$ has a global norm minimizer $e^{\mathbf{s}^*} \in\R^n$, see Section \ref{SubSec:NormMinimizer}. $f$ at $e^{\mathbf{s}^*}$ together with the circuit number $\Theta_f$ equals the Gale dual vector of the support matrix up to a scalar multiple (Corollary \ref{Cor:GaleDual}). Furthermore, it is an immediate consequence of our results that the circuit number is strongly related to the $A$-discriminant of $f$. Particularly, $f \in P_{n,2d} \cap P_{\Delta}^y$ is contained in the topological boundary of the nonnegativity cone, i.e., $f \in \partial (P_{n,2d} \cap P_{\Delta}^y)$, if and only if the $A$-discriminant vanishes at $f$ (Corollary \ref{Cor:ADiscriminant}). These facts about the $A$-discriminant were first shown in \cite{Nie:Discriminants} and \cite{TTdW:genusone}.

$\mathbf{2}$. We consider the case of multiple interior lattice points in the support of $f$. We prove for the case that all coefficients of the interior monomials are negative that all such nonnegative polynomials are in $C_{n,2d}$. Furthermore, we show when such polynomials are sums of squares, again generalizing results in \cite{Fidalgo:Kovacec}.

$\mathbf{3}$. Since the condition of being a sum of squares depends on the combinatorial structure of the simplex $\Delta$, using techniques from toric geometry, we provide sufficient conditions for simplices $\Delta$ such that every nonnegative polynomial in $P_{\Delta}^y$ is a sum of squares, independent from the position of $y \in \Int(\Delta)$. This will prove that for $n = 2$ almost every nonnegative polynomial in $P_{\Delta}^y$ is a sum of squares and this also yields large sections on which nonnegative polynomials and sums of squares coincide.

$\mathbf{4}$. We answer a question of Reznick stated in \cite{Reznick:AGI} whether a certain lattice point criterion on a class of  sparse support sets (more general than circuits) of nonnegative polynomials is equivalent to these polynomials being sums of squares.\\

This article is organized as follows. In Section \ref{Sec:Preliminaries}, we introduce some notations and recall some results that are essential for the upcoming sections and proofs of the main theorems. In Section \ref{Sec:Nonnegative}, we characterize nonnegativity of polynomials $f \in P_{\Delta}^y$. This is done via a norm relaxation method, which is outlined in the beginning of the section. Furthermore, Section \ref{Sec:Nonnegative} deals with invariants and properties of such polynomials and sets them in relation to Gale duals and $A$-discriminants. In Section \ref{Sec:amoebas}, we discuss amoebas of polynomials $f \in P_{\Delta}^y$ and how they are related to nonnegativity respectively the circuit number. In Section \ref{Sec:PSDSOS} we completely characterize the section of the cone of sums of squares with $P_{\Delta}^y$. Furthermore, we generalize results regarding nonnegativity and sums of squares to non-sparse polynomials with simplex Newton polytope. In Section \ref{Sec:convex}, we completely characterize convex polynomials in $P_{\Delta}^y$. In Section \ref{Sec:SOPC}, we provide and discuss a new class of nonnegativity certificate given by \emph{sums of nonnegative circuit polynomials (SONC)}. In Section \ref{Sec:conjecture}, we prove that for non-simplex Newton polytopes $Q$ the lattice point criterion from the simplex case does not suffice to characterize sums of squares. We show that a necessary and sufficient criterion can be given by additionally taking into account the set of possible triangulations of $Q$. This solves an open problem stated by Reznick in \cite{Reznick:AGI}. Finally, in Section \ref{Sec:outlook}, we provide an outlook for future research possibilities.

\section*{Acknowledgments}
We would like to thank Christian Haase for his support and explanations concerning toric ideals and normality. Furthermore, we thank Jens Forsg{\aa}rd, Hannah Markwig, Frank Sottile and Thorsten Theobald for their helpful comments and suggestions on the manuscript. Moreover, we thank the anonymous referees for their helpful comments.

The second author was partially supported by GIF Grant no. 1174/2011 and DFG grant MA 4797/3-2.

\section{Preliminaries}
\label{Sec:Preliminaries}

\subsection{Nonnegative Polynomials and Sums of Squares}
Let $\R[\mathbf{x}]_d = \mathbb R[x_1,\dots,x_n]_{d}$ be the vector space of polynomials in $n$ variables of degree $d$. Denote the convex cone of nonnegative polynomials as
\begin{eqnarray*}
 P_{n,2d} &=& \{p \in \mathbb R[\mathbf{x}]_{2d} : p(\mathbf{x}) \geq 0 \,\,\textrm{ for all }\,\, \mathbf{x}\in \R^n\},
\end{eqnarray*}
and the convex cone of sums of squares as
\begin{eqnarray*}
\Sigma_{n,2d} &=& \left\{p \in P_{n,2d}: p = \sum_{i=1}^k q_i^2 \,\,\textrm{ for }\,\, q_i\in \R[\mathbf{x}]_{d}\right\}.
\end{eqnarray*}

For an introduction of nonnegative polynomials and sums of squares, see \cite{Blekherman:Parrilo:Thomas, Lasserre:Buch, Laurent:Survey}. Since we are interested in nonnegative polynomials and sums of squares in the class $P_{\Delta}^y$, we consider the sections
$$P_{n,2d}^y \ = \ P_{n,2d} \cap P_{\Delta}^y \quad \textrm{ and }\quad \Sigma_{n,2d}^y \ = \ \Sigma_{n,2d} \cap P_{\Delta}^y.$$

\subsection{Amoebas}
\label{SubSec:PreliminariesAmoebas}
For a given Laurent polynomial $f \in \C[z_1,\ldots,z_n]$ on a support set $A \subset \Z^n$ with
variety $\cV(f) \subset (\C^*)^n$, the amoeba $\cA(f)$ is defined as the image of $\cV(f)$ under the log-absolute map $\Log|\cdot|$ defined in \eqref{Equ:LogMap}. Amoebas were first introduced by Gelfand, Kapranov and Zelevinsky in \cite{GKZ:discriminant}. For an overview see \cite{deWolff:Diss,Mikhalkin:Survey,Passare:Tsikh:Survey,Rullgard:Diss}.

\begin{figure}
\ifpictures
$$\includegraphics[width=0.3\linewidth]{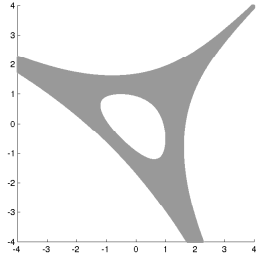}$$
\fi
\caption{The amoeba of the polynomial $f = x_1^2x_2 + x_1x_2^2 - 4x_1x_2 + 1 \in P_\Delta^{(1,1)}$ with $\Delta = \{(0,0),(2,1),(1,2)\}$.}
\label{Fig:Amoeba}
\end{figure}

Amoebas are closed sets \cite{Forsberg:Passare:Tsikh}. Their complements consists of finitely many convex components \cite{GKZ:discriminant}.
Each component of the complement of $\cA(f)$ corresponds to a unique lattice point in $\conv(A) \cap \Z^n$ via an \textit{order map}
\cite{Forsberg:Passare:Tsikh}.

Components of the complement which correspond to vertices of $\conv(A)$ via the order map do always exist. For all other components of the complement of an amoeba $\cA(f)$ the existence depends non trivially on the choice of the coefficients of $f$, see \cite{GKZ:discriminant,Mikhalkin:Survey,Passare:Tsikh:Survey}. We denote the component of the complement of $\cA(f)$ of all points with order $\alp \in \conv(A) \cap \Z^n$ as $E_{\alp}(f)$.\\

\noindent The fiber $\F_{\mathbf{w}}$ of each point $\mathbf{w} \in \R^n$ with respect to the $\Log|\cdot|$-map is given by 
\begin{eqnarray*}
	\F_{\mathbf{w}} & = & \{\mathbf{z} \in (\C^*)^n \ : \ \Log|\mathbf{z}| = \mathbf{w}\}.
\end{eqnarray*}
It is easy to see that $\F_{\mathbf{w}}$ is homeomorphic to a real $n$-torus $(S^1)^n$. For $f = \sum_{\alp \in A} b_\alp \mathbf{z}^{\alp}$ and $\mathbf{v} \in (\C^*)^n$ we define
the \textit{fiber function}
\begin{eqnarray*}
	  f^{|\mathbf{v}|}: (S^1)^n \to \C, \quad  \phi \mapsto
f(e^{\Log|\mathbf{v}| + i \phi}) = \sum_{\alp \in A} b_\alp
\cdot |\mathbf{v}|^\alp \cdot e^{i \langle \alp, \phi \rangle}.
\end{eqnarray*}
This means that $f^{|\mathbf{v}|}$ is the pullback $\vphi_{|\mathbf{v}|}^*(f)$ of $f$ under the homeomorphism $\vphi_{|\mathbf{v}|}: (S^1)^n \to \F_{\Log|\mathbf{v}|} \subset (\C^*)^n$. The crucial fact about the fiber function is that for its zero set $\cV(f^{|\mathbf{v}|})$ it holds that
\begin{eqnarray}
	\cV(f^{|\mathbf{v}|}) & \cong & \cV(f) \cap \F_{\Log|\mathbf{v}|}, \label{Equ:FiberFunction1}
\end{eqnarray}
and hence we have for the amoeba $\cA(f)$ that
\begin{eqnarray}
	\Log|\mathbf{v}| \in \cA(f) & \Lera & \cV(f^{|\mathbf{v}|}) \neq \emptyset. \label{Equ:FiberFunction2}
\end{eqnarray}
For more details on the fiber function see \cite{deWolff:Diss,Mikhalkin:Survey,Schroeter:deWolff:Boundary,TTdW:genusone}.

\subsection{Agiforms}
\label{SubSec:Agiforms}
Asking for nonnegativity of polynomials supported on a circuit is closely related objects called an \emph{agiform} in \cite{Reznick:AGI}. Given a even lattice simplex $\Delta \subset \mathbb R^n$ and an interior lattice point $y \in \Int(\Delta)$, the corresponding agiform to $\Delta$ and $y$ is given by 
$$f(\Delta,\lambda,y) \ = \ \sum_{i=0}^{n} \lambda_i \mathbf{x}^{\alpha(i)} - \mathbf{x}^y$$
where $y = \sum_{i=0}^{n} \lambda_i\alpha(i) \in \N^n$ with $\sum_{i=0}^{n} \lambda_i = 1$ and $\lambda_i\geq 0$. The term \emph{agiform} is implied by the fact that the polynomial $f(\Delta,\lambda,y) = \sum_{i=0}^{n} \lambda_i \mathbf{x}^{\alpha(i)} - \mathbf{x}^y$ is nonnegative by the \emph{arithmetic-geometric mean inequality}. Note that an agiform has a zero at the all ones vector $\mathbf{1}$. This implies that agiforms lie on the boundary of the cone of nonnegative polynomials. A natural question is to characterize those agiforms that can be written as sums of squares. In \cite{Reznick:AGI}, it is shown that this depends non-trivially and exclusively on the combinatorial structure of the simplex $\Delta$ and the location of $y$ in the interior. We need some definitions and results adapted from \cite{Reznick:AGI}.

\begin{defn} Let $\hat\Delta = \{0, \alpha(1), \dots, \alpha(n)\}\subset (2\mathbb N)^n$ be such that  $\conv(\hat \Delta)$ is a simplex and let $L \subseteq \conv(\hat \Delta) \cap \Z^n$.

\begin{enumerate}
 \item  Define $A(L) = \{\frac{1}{2}(s + t) \in \Z^n : s, t\in L \cap (2\Z)^n\}$ and 
$\overline{A}(L) = \{\frac{1}{2}(s + t) \in \Z^n : s\neq t, s, t\in L \cap (2\Z)^n\}$ as the set of averages of even respectively distinct even points in $\conv(L) \cap \Z^n$.
\item   We say that $L$ is $\hat\Delta$-\textit{mediated}, if 
$$\hat\Delta \ \subseteq \ L \ \subseteq \ \overline{A}(L)\cup \hat\Delta,$$ 
i.e., every $\beta \in L\setminus\hat\Delta$ is an average of two distinct even points in $L$.
\end{enumerate}
\end{defn}

\begin{thm}[Reznick \cite{Reznick:AGI}]
There exists a $\hat\Delta$-mediated set $\Delta^*$ satisfying $A(\hat\Delta)\subseteq \Delta^*\subseteq (\Delta\cap \Z^n)$, which contains every $\hat\Delta$-mediated set.
\end{thm}

If $A(\hat\Delta) = \Delta^*$, then we say, motivated by the following example by Reznick, that $\Delta$ is an \emph{$M$-simplex}. Similarly, if $\Delta^* = (\Delta\cap \Z^n)$, then we call $\Delta$ an \emph{$H$-simplex}.

\begin{example}
The standard (Hurwitz-)simplex given by $\conv\{0,2d \cdot e_1,\ldots,2d \cdot e_n\} \subset \R^n$ for $d \in \N$ is an $H$-simplex. The Newton polytope $\conv\{0,(2,4),(4,2)\} \subset \R^2$ of the Motzkin polynomial $f = 1 + x^4y^2 + x^2y^4 - 3x^2y^2$ is an $M$-simplex, see Figure \ref{Fig:Simplices}.
\end{example}

\begin{figure}
\ifpictures
$$\includegraphics[width=0.35\linewidth]{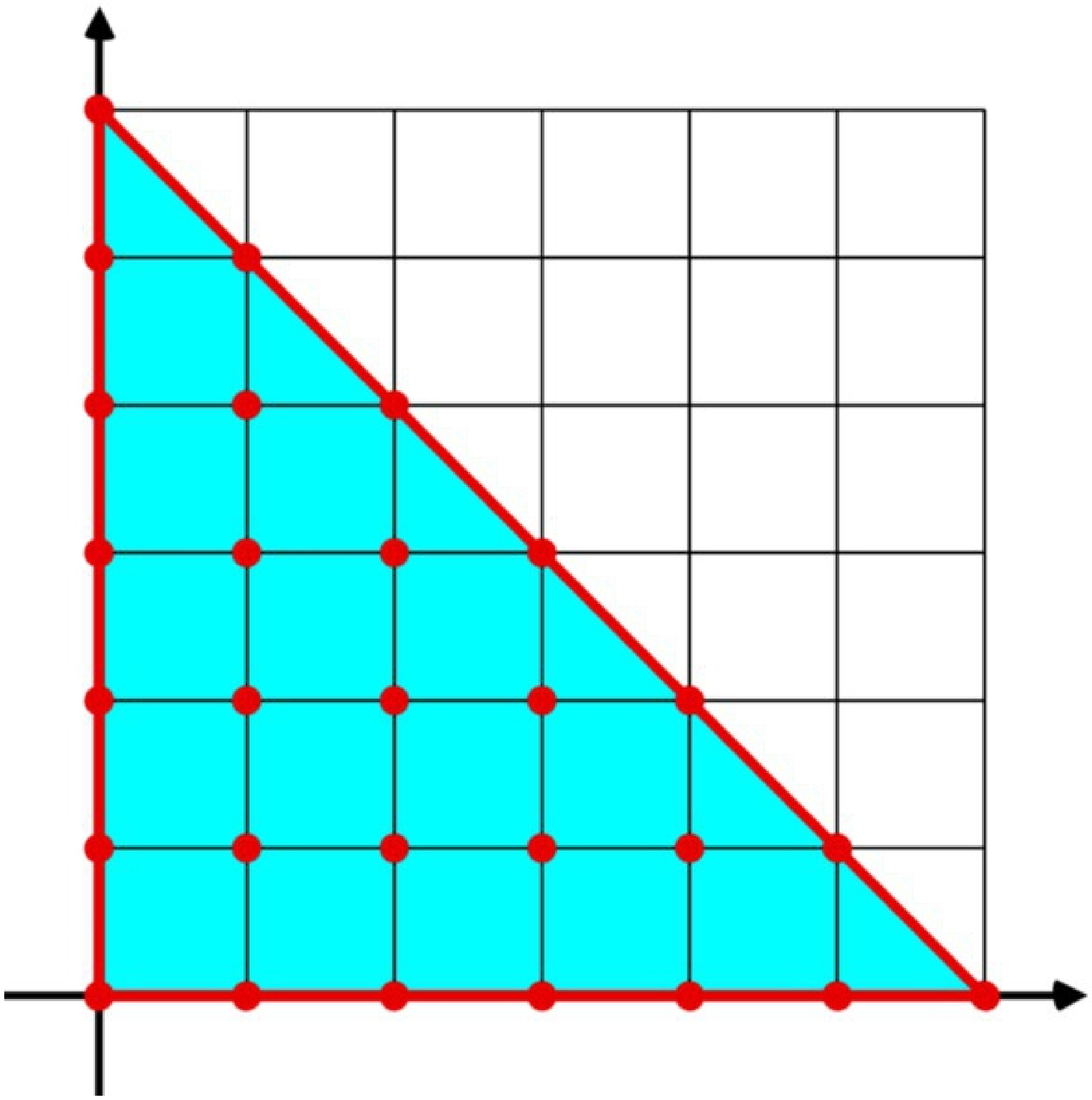} \qquad
\includegraphics[width=0.35\linewidth]{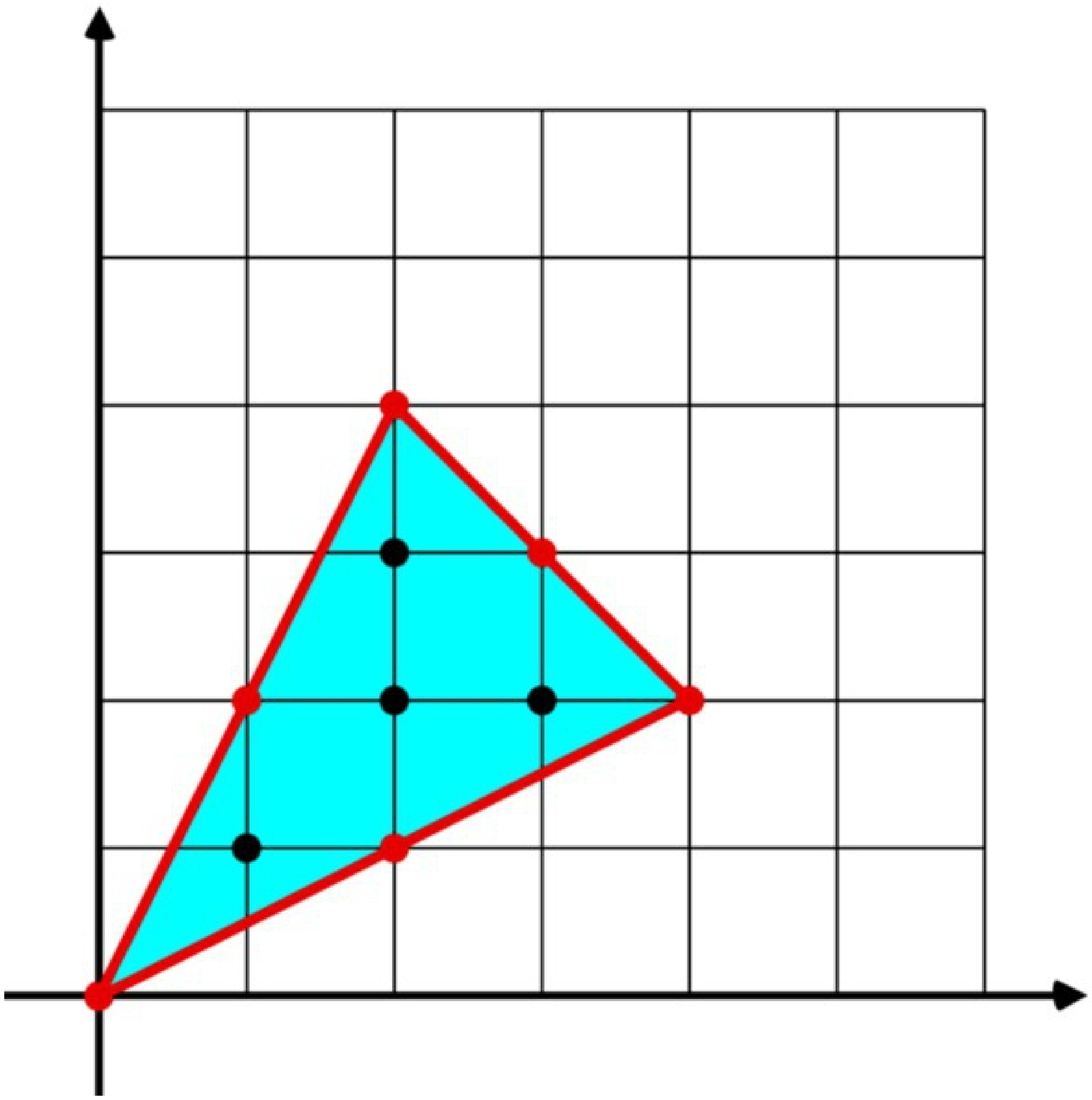}$$
\fi
\caption{On the left: The $H$-simplex $\conv\{(0,0),(6,0),(0,6)\} \subset \R^2$. On the right: The $M$-simplex $\conv\{0,(2,4),(4,2)\} \subset \R^2$. The red (light) points are the lattice points contained in the corresponding sets $\Delta^*$.}
\label{Fig:Simplices}
\end{figure}

The main result in \cite{Reznick:AGI} concerning the question under which conditions agiforms are sums of squares is given by the following theorem.
\begin{thm}[Reznick \cite{Reznick:AGI}]
\label{thm:reznick}
\label{thm:rezmain}
Let $f(\Delta,\lambda,y)$ be an agiform. Then $f(\Delta,\lambda,y) \in \Sigma_{n,2d}$ if and only if $y \in \Delta^*$.
\end{thm}

\section{Invariants and Nonnegativity of Polynomials Supported on Circuits}
\label{Sec:Nonnegative}

The main contribution of this section is the characterization of $P_{n,2d}^y$, i.e., the set of nonnegative polynomials supported on a circuit (Theorem \ref{Thm:Positiv}). Along the way we provide standard forms and invariants, which reflect the nice structural properties of the class $P_{\Delta}^y$.

In Section \ref{SubSec:NormMinimizerStrategy} we outline the norm relaxation method, which is the proof method used for the characterization of nonnegativity. In Section \ref{SubSec:NormMinimizer}, we introduce standard forms for polynomials in $P_{\Delta}^y$ and, in particular, prove the existence of a particular norm minimizer for polynomials, where the coefficient $c$ equals the negative circuit number $\Theta_f$ (Proposition \ref{Prop:GlobalMinimizer}).
In Section \ref{SubSec:Nonnegative}, we put all pieces together and characterize nonnegativity of polynomials in $P_{\Delta}^y$ (Theorem \ref{Thm:Positiv}). In Section \ref{SubSec:GaleDual}, we discuss connections to Gale duals and $A$-discriminants.

\subsection{Nonnegativity via Norm Relaxation}
\label{SubSec:NormMinimizerStrategy}

We start with a short outline of the proof method, which we introduce and apply here in order to tackle the problem of nonnegativity of polynomials. Let $f = \sum_{\alp \in A} b_\alp \mathbf{x}^{\alp} \in \R[\mathbf{x}]$ be a polynomial with $A \subset \N^n$ finite, $0 \in A$ and $\alp \in (2\N)^n$ as well as $b_\alp > 0$ if $\alp$ is contained in the vertex set $\V(A)$ of $\conv(A)$. Instead of trying to answer the question whether $f(\mathbf{x}) \geq 0$ for all $\mathbf{x} \in \R^n$, we investigate the relaxed problem
\begin{eqnarray}
	\text{Is } f(|\mathbf{x}|) & = & \sum_{\alp \in \V(A)} b_\alp \cdot |\mathbf{x}^{\alp}| - \sum_{\alp \in A \setminus \V(A)} |b_\alp| \cdot |\mathbf{x}^{\alp}| \geq 0 \ \text{ for all } \ \mathbf{x} \in \R^n_{\geq 0} \text{ ?} \label{Equ:NormRelaxation}
\end{eqnarray}
Since $b_\alp \cdot |\mathbf{x}^{\alp}| = b_\alp \cdot \mathbf{x}^{\alp}$ for $\alp \in V(A)$ and $-b_\alp \cdot |\mathbf{x}^{\alp}| \leq b_\alp \cdot \mathbf{x}^{\alp}$ for $\alp \in A \setminus \V(A)$ we have $f(|x|) \leq f(x)$.

Since the strict positive orthant $\R^n_{> 0}$ is an open dense set in $\R^n_{\geq 0}$ and the componentwise exponential function $\Exp: \R^n \to \R_{> 0}^n, (x_1,\ldots,x_n) \mapsto (\exp(x_1),\ldots,\exp(x_n))$ is a bijection, Problem \eqref{Equ:NormRelaxation} is equivalent to the question
\begin{eqnarray}
	\text{Is } f(e^{\mathbf{w}}) & = & \sum_{\alp \in \V(A)} b_\alp \cdot e^{\langle \mathbf{w},\alp \rangle} - \sum_{\alp \in A \setminus \V(A)} |b_\alp| \cdot e^{\langle \mathbf{w},\alp \rangle} \geq 0 \ \text{ for all } \ \mathbf{w} \in \R^n \text{ ?} \label{Equ:ExpSumRelaxation}
\end{eqnarray}
Hence, an affirmative answer of \eqref{Equ:ExpSumRelaxation} implies nonnegativity of $f$. The motivation for the relaxation is that, on the one hand, Question \eqref{Equ:ExpSumRelaxation} is eventually easier to answer, since we have \textit{linear} operations on the exponents and, on the other hand, the gap between \eqref{Equ:ExpSumRelaxation} and nonnegativity hopefully is not too big, in particular for \textit{sparse} polynomials. We show that for polynomials supported on a circuit (and some more general classes of sparse polynomials) both is true: In fact, for circuit polynomials the question of nonnegativity and \eqref{Equ:ExpSumRelaxation} is \textit{equivalent} and can be characterized exactly, explicitly, and easily in terms of the coefficients of $f$ and the combinatorial structure of $A$.

An interesting side effect of the described relaxation is that \eqref{Equ:ExpSumRelaxation} is strongly related to the \textit{amoeba} of $f$ as we point out (for circuit polynomials) in the following Section \ref{Sec:amoebas}. Thus, it will serve us as a bridge between real algebraic geometry and amoeba theory.

\subsection{Standard Forms and Norm Minimizers of Polynomials Supported on Circuits}
\label{SubSec:NormMinimizer}

Let $f$ be a polynomial of the Form (\ref{Equ:OurPolynomials}) defined on a circuit $A = \{\alp(0),\ldots,\alp(n),y\}$ $\subset \Z^n$. Observe that there exists a unique convex combination $\sum_{j = 0}^n \lam_j \alp(j) = y$. In the following, we assume without loss of generality that $\alp(0) = 0$, which is possible, since we can factor out a monomial $\mathbf{x}^{\alp(0)}$ with $\alp(0) \in (2\N)^n$ if necessary. 
We define the \textit{support matrix} $M^A$ by
\begin{eqnarray*}
	M^A & = & \left(
\begin{array}{ccccc}
1 & 1 & \cdots & 1 & 1 \\
0 & \alp(1)_1 & \cdots & \alp(n)_1 & y_1 \\
\vdots & \vdots & \ddots & \vdots & \vdots \\
0 & \alp(1)_n & \cdots & \alp(n)_n & y_n \\
\end{array}
\right) \ \in \ \Mat(\Z,(n+1) \times (n+2)),
\end{eqnarray*}
and $M^A_j$ as the matrix obtained by deleting the $j$-th column of $M^A$, where we start to count at 0. Furthermore, we always assume that $b_0 = \lam_0$, which is always possible, since multiplication with a positive scalar does not affect if a polynomial is nonnegative. We denote the canonical basis of $\R^n$ with $e_1,\ldots,e_n$.

\begin{prop}
Let $f$ be of the Form \eqref{Equ:OurPolynomials} supported on a circuit $A = \{\alp(0),\ldots,\alp(n),y\}$ $\subset \Z^n$ and $y = \sum_{j = 0}^n \lam_j \alp(j)$ with $\sum_{j = 0}^n \lam_j = 1$, $0 < \lam_j < 1$ for all $j$. Let $\mu \in \N_{> 0}$ denote the least common multiple of the denominators of the $\lam_j$. Then there exists a unique polynomial $g$ of the Form \eqref{Equ:OurPolynomials} with $\supp(g) = A' = \{0,\alp(1)',\ldots,\alp(n)',y'\} \subset \Z^n$ such that the following properties hold.
\begin{enumerate}
	\item $M^A = \left(\begin{array}{cc} 1 & 0 \\ 0 & T \\ \end{array}\right) M^{A'}$ for some $T \in GL_{n}(\Q)$,
	\item $f$ and $g$ have the same coefficients,
	\item $\alp(j)' = \mu \cdot e_j$ for every $1 \leq j \leq n$,
	\item $y' = \sum_{j = 1}^n \lam_j \alp(j)'$,
	\item $f(e^{\mathbf{w}}) = g(e^{T^t\mathbf{w}})$ for all $\mathbf{w} \in \R^n$.
\end{enumerate}
\label{Prop:StandardForm}
\end{prop}

For every $f$ of the Form (\ref{Equ:OurPolynomials}) we call the polynomial $g$, which satisfies all the conditions of the proposition, the \textit{standard form} of $f$. Note that $f(e^{\mathbf{w}})$ is defined in the sense of \eqref{Equ:ExpSumRelaxation} and the support matrix $M^{A'}$ of the standard form of $f$ is of the shape
\begin{eqnarray}
	M^{A'} & = & \left(
\begin{array}{cccccc}
1 & 1 & \cdots & \cdots & 1 & 1 \\
0 & \mu & 0 & \cdots & 0 & \mu \lam_1\\
\vdots & 0 & \ddots & & \vdots & \vdots \\
\vdots & \vdots & & \ddots & 0 & \vdots \\
0 & 0 & \cdots & 0 & \mu & \mu \lam_n\\
\end{array}\right)
\ \in \ \Mat(\Z,(n+1) \times (n+2)). \label{Equ:StandardFormSupportMatrix}
\end{eqnarray}

\begin{proof}
We assume without loss of generality that $\alp(0) = 0$. Let $\ovl M^A_{n+1}$ be the submatrix of $M^A_{n+1}$ obtained by deleting the first row and column; analogously for $\ovl M^{A'}_{n+1}$. By definition, we have $\alp(j) = \ovl M^A_{n+1} e_j$ and $\alp(j)' = \ovl M^{A'}_{n+1} e_j$ for $1 \leq j \leq n$. We construct the polynomial $g$. We choose the same coefficients for $g$ as for $f$. Since $0,\alp(1),\ldots,\alp(n)$ form a simplex, there exists a unique matrix $T \in GL_{n}(\Q)$ such that
\begin{eqnarray*}
	M^A_{n+1} & = & \left(\begin{array}{cc} 1 & 0 \\ 0 & T \\ \end{array}\right) M^{A'}_{n+1}
\end{eqnarray*}
with $M^{A'}$ of the Form \eqref{Equ:StandardFormSupportMatrix} given by $\mu T = (\ovl M^A_{n+1})^{-1}$. Since $y = \sum_{j = 0}^n \lam_j \alp(j)$, it follows that, in affine coordinates, we have $y'_j e_j = T^{-1}\lam_j (\ovl M^A_{n+1} e_j)$, i.e., $y' = \mu(\lam_0,\ldots,\lam_n)$. Thus, (1) -- (4) holds.

We show that $f(e^{\mathbf{w}}) = g(e^{T^t \mathbf{w}})$ for every $\mathbf{w} \in \R^n$. We investigate the monomial $\mathbf{x}^{\alp(j)}$:
\begin{eqnarray*}
	b_j e^{\langle \alp(j), \mathbf{w} \rangle} \ = \ b_j e^{\langle \ovl  M^{A}_{n+1} e_j, \mathbf{w} \rangle} \ = \ b_j e^{\langle T  \ovl M^{A'}_{n+1} e_j , \mathbf{w} \rangle} \ = \ b_j e^{\langle \alp(j)', T^t \mathbf{w} \rangle}
\end{eqnarray*}
For the inner monomials $y$ and $y'$ we know that $y = T y'$ and thus for $y' = \sum_{j = 0}^n \lam_j \alp(j)'$ we have $y = T (\sum_{j = 0}^n \lam_j \alp(j)') = \sum_{j = 0}^n \lam_j T \alp(j)' = \sum_{j = 0}^n \lam_j \alp(j)$. Therefore, (5) follows from
\begin{eqnarray*}
c e^{\langle y,\mathbf{w} \rangle} \ = \ c e^{\langle \sum_{j = 0}^n \lam_j \alp(j),\mathbf{w} \rangle} \ = \ c e^{\sum_{j = 0}^n \lam_j \langle \alp(j),\mathbf{w} \rangle} \ = \ c e^{\sum_{j = 0}^n \lam_j \langle \alp(j)',T^t \mathbf{w} \rangle} \ = \ c e^{\langle y',T^t \mathbf{w} \rangle}.
\end{eqnarray*}
\end{proof}

Proposition \ref{Prop:StandardForm} can easily be generalized to polynomials
\begin{eqnarray}
	f & = & b_0 + \sum_{j = 1}^n b_j \mathbf{x}^{\alp(j)} + \sum_{y(i) \in I} a_i \mathbf{x}^{y(i)} \in \R[\mathbf{x}], \label{Equ:GeneralSimplexf}
\end{eqnarray}
with $\New(f) = \Delta = \conv\{0,\alp(1),\ldots,\alp(n)\}$ being a simplex and $I \subset (\Int(\Delta) \cap \Z^n)$. Every $y(i)$ has a unique convex combination $y(i) = \lam_0^{(i)} + \sum_{j = 1}^n \lam_j^{(i)} \alp(j)$ with $\lam_j^{(i)} > 0$ for all $i,j$.

\begin{cor}
Let $f$ be defined as in \eqref{Equ:GeneralSimplexf}. Then Proposition \ref{Prop:StandardForm} holds literally if we apply (4) for every $y(i)$ and define $\mu$ as the least common multiple of the denominators of all $\lam_j^{(i)}$.
\label{Cor:StandardFormGeneralSimplex}
\end{cor}

\begin{proof}
By definition of $\mu$, the support matrix $M^{A'}$ is integral again. Since in the proof of Proposition \ref{Prop:StandardForm} neither uniqueness of $y$ is used nor special assumptions about $y$ were made, the statement follows.
\end{proof}

Now, we return to the case of circuit polynomials.

\begin{prop}
Let $f = \lam_0 + \sum_{j = 1}^n b_j \mathbf{x}^{\alp(j)} + c \mathbf{x}^{y} \in P_\Delta^y$ be such that $c < 0$ and $y = \sum_{j = 1}^n \lam_j \alp(j)$ with $\sum_{j = 0}^n \lam_j = 1$, $\lambda_j\geq 0$. Then $f(e^{\mathbf{w}})$ with $\mathbf{w} \in \R^n$ has a unique extremal point, which is always a minimum.
\label{Prop:ExtremalPoint}
\end{prop}

This proposition was used in \cite{TTdW:genusone} (see Lemma 4.2 and Theorem 5.4). For convenience, we give an own, easier proof here.

\begin{proof}
We investigate the standard form $g$ of $f$.  For the partial derivative $x_j \partial g / \partial x_j$ (we can multiply with $x_j$, since $e^{\mathbf{w}} \geq 0$) we have
\begin{eqnarray*}
	x_j \frac{\partial g}{\partial x_j} = b_j \mu x_j^{\mu - 1} + c \lam_j \mu x_j^{\lam_j \mu - 1} \prod_{k = 2}^n x_k^{\lam_k \mu}.
\end{eqnarray*}
Hence, the partial derivative vanishes for some $e^{\mathbf{w}}$ if and only if
\begin{eqnarray*}
	\exp \left(w_j \mu - \sum_{k = 1}^n \lam_k \mu w_k\right) & = & - \frac{c \lam_j}{b_j}.
\end{eqnarray*}
Since the right hand side is strictly positive, we can apply $\log|\cdot|$ on both sides for every partial derivative and obtain the following linear system of equations
\begin{eqnarray*}
	\left(E_n - \left(\begin{array}{ccc} \lam_1 & \cdots & \lam_n \\ \vdots & \ddots & \vdots \\ \lam_1 & \cdots & \lam_n \\ \end{array}\right) \right) \cdot \left(\begin{array}{c}  w_1 \\ \vdots \\ w_n \\ \end{array}\right) & = & 
 \left(\begin{array}{c}  1 / \mu (\log(\lam_1) + \log(-c) - \log(b_1)) \\ \vdots \\ 1 / \mu (\log(\lam_n) + \log(-c) - \log(b_n)) \\ \end{array}\right).
\end{eqnarray*}
Since the matrix on the left hand side has full rank, we have a unique solution.\\

For arbitrary $f$ we have $f(e^{\mathbf{w}}) = g(e^{T^t \mathbf{w}})$ by Proposition \ref{Prop:StandardForm} and, hence, if $\mathbf{w}^*$ is the unique extremal point for $g(e^{\mathbf{w}})$, then $(T^t)^{-1} \mathbf{w}^*$ is the unique extremal point for $f(e^{\mathbf{w}})$.

For every $\mathbf{w} \in \R^n$ with $||\mathbf{w}|| \to \infty$ the polynomial $f$ converges against the terms with exponents which are contained in a particular proper face of $\New(f)$. Since all these terms are strictly positive, $f(e^{\mathbf{w}})$ converges against a number in $\R_{> 0} \cup \{\infty\}$. Thus, the unique extremal point has to be a global minimum.
\end{proof}

For $f \in P_{\Delta}^y$ we define $\mathbf{s}^*_f \in \R^n$ as the unique vector satisfying
\begin{eqnarray*}
	\prod_{k = 1}^n (e^{s_{k,f}^*})^{\alp(j)_k} \ = \ e^{\langle \mathbf{s}^*_f,\alp(j) \rangle} & = & \frac{\lam_j}{b_j} \ \text{ for all } \ 1 \leq j \leq n.
\end{eqnarray*}
$\mathbf{s}^*_f$ indeed is well defined, since application of $\log|\cdot|$ on both sides yields a linear system of equations with variables $s_{k,f}^*$ and the rank of this system has to be $n$, since $\conv(A)$ is a simplex. If the context is clear, then we simply write $\mathbf{s}^*$ instead of $\mathbf{s}^*_f$ and $e^{\mathbf{s}^*}$ instead of $e^{\mathbf{s}^*_f}$. We recall that the circuit number associated to a polynomial $f \in P_{\Delta}^y$ is given by $\Theta_f = \prod_{j = 0}^n \lf(\frac{b_j}{\lam_j}\ri)^{\lam_j} = \prod_{j = 1}^n \lf(\frac{b_j}{\lam_j}\ri)^{\lam_j}$. 

\begin{prop}
For $f \in P_{\Delta}^y$ and $c = -\Theta_f$ the point $\mathbf{s}^* \in \R^n$ is a root and the unique global minimizer of $f(e^{\mathbf{w}})$.
\label{Prop:GlobalMinimizer}
\end{prop}

Due to this proposition we call the point $\mathbf{s}^*$ the \textit{norm minimizer} of $f$. We remark that this proposition was already shown for polynomials in $P_{\Delta}^y$ in standard form in \cite{Fidalgo:Kovacec} and for arbitrary simplices but in a more complicated way in \cite{TTdW:genusone}.

\begin{proof}
For $f(e^{\mathbf{s}^*})$ we have
\begin{eqnarray*}
	f\left(e^{\mathbf{s}^*}\right) & = & \lam_0 + \sum_{j = 1}^n b_j e^{\langle \mathbf{s}^*, \alp(j) \rangle} - \Theta_f e^{\langle \mathbf{s}^*, y \rangle} 
 \ = \ \sum_{j = 0}^n \lam_j - \Theta_f \cdot\prod_{j = 1}^n \left(\frac{\lam_j}{b_j} \right)^{\lam_j} \ = \ 1 - 1 \ = \ 0.
\end{eqnarray*}

For the minimizer statement, we investigate the partial derivatives $x_j \partial f / \partial x_j$ (we can multiply with $x_j$, since $e^{\mathbf{w}} > 0$). Since $y_j = \sum_{k = 1}^n \lam_j \alp_j(k)$, we obtain
\begin{eqnarray*}
	x_j \frac{\partial f}{\partial x_j} & = & \sum_{k = 1}^n b_k \alp_j(k) \mathbf{x}^{\alp(k)} - \Theta_f \cdot \left(\sum_{k = 1}^n \lam_j \alp_j(k)\right) \mathbf{x}^{y}.
\end{eqnarray*}

Evaluation of the partial derivative at $e^{\mathbf{s}^*}$ yields
\begin{eqnarray*}
	x_j \frac{\partial f}{\partial x_j}(e^{\mathbf{s}^*}) & = & \sum_{k = 1}^n b_k \alp_j(k) \left(\frac{\lam_k}{b_k}\right) - \Theta_f \left(\sum_{k = 1}^n \lam_j \alp_j(k)\right) \cdot\prod_{j = 1}^n \left(\frac{\lam_j}{b_j} \right)^{\lam_j} \\
	& = & \sum_{k = 1}^n \lam_j \alp_j(k) - \sum_{k = 1}^n \lam_j \alp_j(k) \ = \ 0.
\end{eqnarray*}

Finally, by Proposition \ref{Prop:ExtremalPoint}, $e^{\mathbf{s}^*}$ is the unique global minimizer of $f(e^{\mathbf{w}})$.
\end{proof}

In some contexts it is more convenient to work with a Laurent polynomial supported on a circuit where the interior point $y$ equals the origin. With the same argumentation as before we find a suitable standard form.

\begin{cor}
Let $f$ and all notations be as in Proposition \ref{Prop:StandardForm}. Then there exists a unique Laurent polynomial $g$ of the Form (\ref{Equ:OurPolynomials}) with $\supp(g) = A'' = \{\alp(0)'',\ldots,\alp(n)'',0\} \subset \Z^n$ such that the following properties hold:
\begin{enumerate}
	\item $M_A = \left(\begin{array}{cc} 1 & 0 \\ 0 & T \\ \end{array}\right) M_{A''}$ for some $T \in GL_n(\Q)$,
	\item $f$ and $g$ have the same coefficients,
	\item $\alp(j)'' = \mu \cdot e_j$ for every $1 \leq j \leq n$,
	\item $\sum_{j = 0}^n \lam_j \alp(j)'' = 0$,
	\item $f(e^{\mathbf{w}}) = g(e^{T^t\mathbf{w}})$ for all $\mathbf{w} \in \R^n$.
\end{enumerate}
\label{Cor:StandardZeroForm}
\end{cor}

For every polynomial $f$ of the Form (\ref{Equ:OurPolynomials}), we call the polynomial $g$, which satisfies all conditions in Corollary \ref{Cor:StandardZeroForm}, the \textit{zero standard form} of $f$. Note that the support matrix $M^{A''}$ of the zero standard form of $f$ is of the shape
\begin{eqnarray}
	M^{A''} & = & \left(
\begin{array}{cccccc}
1 & 1 & \cdots & \cdots & 1 & 1 \\
- \frac{\lam_1 \mu}{\lam_0} & \mu & 0 & \cdots & 0 & 0\\
\vdots & 0 & \ddots & & \vdots & \vdots \\
\vdots & \vdots & & \ddots & 0 & \vdots \\
- \frac{\lam_n \mu}{\lam_0} & 0 & \cdots & 0 & \mu & 0\\
\end{array}\right)
\ \in \ \Mat(\Z,(n+1) \times (n+2)). \label{Equ:ZeroStandardFormSupportMatrix}
\end{eqnarray}

\begin{proof}
We divide $f$ by $\mathbf{x}^{y}$, which is always possible, since $e^{\mathbf{w}} > 0$. We apply literally the proof of Proposition \ref{Prop:StandardForm} with the exception of using the matrix $M_0^A$ instead of $M_{n+1}^A$ and the convex combination $-\lam_0 \alp(0) = \sum_{j = 1}^n \lam_j \alp(j)$ instead of $y = \sum_{j = 0}^n \lam_j \alp(j)$.
\end{proof}

An advantage of the zero standard form is that the global minimizer does not longer depend on the choice of $c$.

\begin{cor}
For $f \in P_{\Delta}^y$ the point $e^{\mathbf{s}^*}$ is a global minimizer for $(f / \mathbf{x}^y)(e^{\mathbf{w}})$ independent of the choice of $c$.
\label{Cor:GlobalMinimizer}
\end{cor}

\begin{proof}
By Corollary \ref{Cor:StandardZeroForm}, we can transform $f$ into zero standard form with $y = 0$. Then the proof of Proposition \ref{Prop:GlobalMinimizer} can be literally applied again with the exception of $(f/\mathbf{x}^y)(e^{\mathbf{w}}) = 0$ if and only if $c = -\Theta_f$.
\end{proof}

\subsection{Nonnegativity of Polynomials Supported on a Circuit}
\label{SubSec:Nonnegative}

In this section, we characterize nonnegativity of polynomials in $P_\Delta^y$. The following lemma allows us to reduce the case of $y \in \partial\Delta$ to the case $y \in \Int(\Delta)$.

\begin{lemma}
 Let $f = b_0 + \sum_{j=1}^n b_j\mathbf{x}^{\alpha(j)} + c\cdot \mathbf{x}^y$ be such that the Newton polytope is given by $\Delta = \New(f) = \conv\{0,\alpha(1),\dots,\alpha(n)\}$ and $y \in \partial\Delta$. Furthermore, let $F$ be the face of $\Delta$ containing $y$. Then $f$ is nonnegative if and only if the restriction of $f$ to the face $F$ is nonnegative.
\end{lemma}

\begin{proof}
 For the necessity of nonnegativity of the restricted polynomial, see \cite{Reznick:AGI}. Otherwise, the restriction to the face $F$ contains the monomial $\mathbf{x}^y$ and this restriction is nonnegative. Since all other terms in $f$ correspond to the (even) vertices of $\Delta$ and have nonnegative coefficients, the claim follows.
\end{proof}

Now, we show the first part of our main Theorem \ref{Thm:MainEquivalences} by characterizing nonnegative polynomials $f \in P_{\Delta}^y$ supported on a circuit. Recall that we denote such polynomials of degree $2d$ in in $n$ variables as $P_{n,2d}^y$. Note that this theorem covers the known special cases of agiforms \cite{Reznick:AGI} and circuit polynomials in standard form \cite{Fidalgo:Kovacec}.

\begin{thm}
 Let $f =  \lambda_0 + \sum_{j=1}^n b_j\mathbf{x}^{\alpha(j)} + c\cdot \mathbf{x}^y \in P_{\Delta}^y$ be of the Form \eqref{Equ:OurPolynomials} with $\alp(j) \in (2\N)^n$. Then the following are equivalent.
\begin{enumerate}
 \item $f \in P_{n,2d}^y$, i.e., $f$ is nonnegative.
\item $|c| \leq \Theta_f$ and $y \notin (2\mathbb N)^n$ or $c \geq - \Theta_f$ and $y \in (2\mathbb N)^n$. 
\end{enumerate}
\label{Thm:Positiv}
\end{thm}

\begin{proof} First, observe that $f \geq 0$ is trivial for $c \geq 0$ and $y \in (2\mathbb N)^n$, since in this case $f$ is a sum of monomial squares.

We apply the norm relaxation strategy introduced in Section \ref{SubSec:NormMinimizerStrategy}. Initially, we show that $f(\mathbf{x}) \geq 0$ if and only if $f(e^{\mathbf{w}}) \geq 0$ for all $f \in P_{\Delta}^y$.  Let without loss of generality $y_1,\ldots,y_k$ be the odd entries of the exponent vector $y$. Thus, for every $1 \leq j \leq k$ replacing $x_j$ by $-x_j$ changes the sign of the term $c\cdot \mathbf{x}^y$. Since all other terms of $f$ are nonnegative for every choice of $\mathbf{x} \in \R^n$, we have $f(\mathbf{x}) \geq 0$ if $\sgn(c) \cdot \sgn(x_1) \cdots \sgn(x_k) = 1$. Since furthermore, for $\sgn(c) \cdot \sgn(x_1) \cdots \sgn(x_k) = -1$ we have $c\cdot \mathbf{x}^y = -|c| \cdot |x_1|^{y_1} \cdots |x_n|^{y_n}$, we can assume $c \leq 0$ and $\mathbf{x} \geq 0$ without loss of generality. Then $\lam_0 + \sum_{j = 0}^n b_j \mathbf{x}^{\alp(j)} - |c| |\mathbf{x}|^{y}$ is nonnegative for all $\mathbf{x} \in \R^n$ if and only if this is the case for all $\mathbf{x} \in \R_{\geq 0}^n$. And since $\R_{> 0}^{n}$ is an open, dense set in $\R_{\geq 0}^{n}$, we can restrict ourselves to the strict positive orthant. With the componentwise bijection between $\R_{> 0}^n$ and $\R^n$ given by the $\Exp$-map, it follows that $f(\mathbf{x}) \geq 0$ for all $\mathbf{x} \in \R^n$ if and only if $f(e^{\mathbf{w}}) \geq 0$ for all $\mathbf{w} \in \R^n$. Hence, the theorem is shown if we prove that$f(e^{\mathbf{w}}) \geq 0$ for all $\mathbf{w} \in \R^n$ if and only if $c \in [-\Theta_f,0]$.\\

We fix some arbitrary $b_1,\ldots,b_n \in \R_{> 0}$ and denote by $(f_c)_{c \in \R}$ be the corresponding family of polynomials in $P_{\Delta}^y$. By Proposition \ref{Prop:GlobalMinimizer}, $f_c(e^{\mathbf{w}})$ has a unique global minimum for $c = -\Theta_f$ attained at $\mathbf{s}^* \in \R^n$ satisfying $f_{-\Theta_f}(e^{\mathbf{s}^*}) = 0$. Since $e^{\mathbf{s}^*}$ is a \textit{global} (norm) minimum, this implies, in particular, $f_c(e^{\mathbf{w}}) \geq 0$ for all $\mathbf{w} \in \R^n$ if $c = -\Theta_f$.

But this fact also completes the proof for general $c < 0$: Since $c \cdot e^{\langle \mathbf{w}, y \rangle}$ is the unique negative term in $f_c(e^{\mathbf{w}})$ for all $\mathbf{w} \in \R^n$, a term by term inspection yields that $f_c(e^{\mathbf{w}}) < f_{-\Theta_f}(e^{\mathbf{w}})$ if and only if $c < -\Theta_f$. Hence, $f_c(e^{\mathbf{w}}) < 0$ for some $\mathbf{w} \in \R^n$ if and only if $c < -\Theta_f$.
\end{proof}

An immediate consequence of the theorem is an upper bound for the number of zeros of polynomials $f \in \partial P_{n,2d}^y$.

\begin{cor}\label{cor:zerobound}
Let $f \in \partial P_{n,2d}^y$. Then $f$ has at most $2^n$ affine real zeros $\mathbf{v} \in \R^n$, which all satisfy $|x_j| = e^{s_j^*}$ for all $1 \leq j \leq n$.
\end{cor}

\begin{proof}
Assume $f \in \partial P_{n,2d}^y$ and $f(\mathbf{x}) = 0$ for some $\mathbf{x} \in \R^n$. Then we know by the proof of Theorem \ref{Thm:Positiv} that $|x_j| = e^{s_j^*}$. Thus, $\mathbf{x} = (\pm e^{s_1^*},\ldots,\pm e^{s_n^*})$.
\end{proof}

The bound in Corollary \ref{cor:zerobound} is sharp as demonstrated by the well-known Motzkin polynomial $f = 1 + x_1^2x_2^4 + x_1^4x_2^2 - 3x_1^2x_2^2 \in P_{2,6}^y$. The zeros are given by $\mathbf{x} = (\pm 1, \pm 1)$. Furthermore, it is important to note that the maximum number of zeros does not depend on the degree of the polynomials, which is in sharp contrast to previously known results concerning the maximum number of zeros of nonnegative polynomials and sums of squares, \cite{Reznick:realzeros}.\\

In order to illustrate the results of this section, we give an example. Let $f = 1 + x_1^2x_2^4 + x_1^4x_2^2 - 3x_1^2x_2^2$ be the Motzkin polynomial. $f$ is supported on a circuit $A$ with $y = \sum_{j = 0}^2 \frac{1}{3} \alp(j)$. We apply Proposition \ref{Prop:StandardForm} and compute the standard form $g$ of $1/3 \cdot f$. Then $g$ is the polynomial, which is supported on a circuit $A' = \{0,\alp(1)',\alp(2)'\}$ satisfying $M^A = \left(\begin{array}{cc} 1 & 0 \\ 0 & T \\ \end{array}\right) M^{A'}$ for some $T \in GL_{n}(\Q)$ with $\alp(1)' = (\mu,0)^t$, $\alp(2)' = (0,\mu)^t$ and $y' =  1/3 \alp(1)' + 1/3 \alp(2)'$, where $\mu = \lcm\{1/\lam_0,1/\lam_1,1/\lam_2\} = \lcm\{3,3,3\} = 3$. Additionally, $g$ has the same coefficients as $f$. It is easy to see that
\begin{eqnarray*}
	T & = & \left(\begin{array}{cc} 4/3 & 2/3 \\ 2/3 & 4/3 \\ \end{array}\right)
\end{eqnarray*}
and thus
\begin{eqnarray*}
	g & = & 1/3 + 1/3 x_1^3 + 1/3 x_2^3 - x_1 x_2
\end{eqnarray*}
and, by Proposition \ref{Prop:StandardForm} we have $f(e^{\mathbf{w}}) = g(e^{T^t \mathbf{w}})$.

Since the circuit number $\Theta_f$ only depends on the coefficients of $f$ and the convex combination of $y$, it is invariant with respect to transformation to the standard form. Thus, we have
\begin{eqnarray*}
\Theta_f \ = \ \Theta_g  \ = \ \prod_{j = 0}^2 \left(\frac{\lam_j}{b_j}\right)^{\lam_j} \ = \ \left(\frac{1/3}{1/3}\right)^{1/3} \cdot \left(\frac{1/3}{1/3}\right)^{1/3} \cdot \left(\frac{1/3}{1/3}\right)^{1/3} \ = \ 1.
\end{eqnarray*}

Since $y = (2,2) \in (2\N)^2$, by Theorem \ref{Thm:Positiv}, $f \geq 0$ if and only if the inner coefficient $c$ of $f$ satisfies $c \geq -\Theta_f = -1$. But the inner coefficient $c$ of the Motzkin polynomial equals its negative circuit number. Hence, the Motzkin polynomial is contained in the boundary of the cone of nonnegative polynomials.

If $c = -\Theta_f$, then we know by Proposition \ref{Prop:GlobalMinimizer} that $f(e^{\mathbf{w}}) = 0$ at the unique point $\mathbf{s}^*$ with
\begin{eqnarray*}
	1/3 \cdot e^{4s_1^* + 2s_2^*} \ = \ 1/3 \quad \text{ and } \quad 1/3 \cdot e^{2s_1^* + 4s_2^*} \ = \ 1/3.
\end{eqnarray*}
Thus, $\mathbf{s}^* = (0,0)$. Since, by the proof of Theorem \ref{Thm:Positiv}, $f(\mathbf{x}) = 0$ only if $f(|x_1|,|x_2|) = 0$, we can conclude that every affine root $\mathbf{v} \in \R^n$ of the Motzkin polynomial satisfies $|v_j| = 1$.\\

We give a second example where nonnegativity is not a priori known. Let $f = 1/4 + 2 \cdot x_1^2x_2^4 + x_1^4x_2^4 - 2.5 \cdot x_1^2x_2^3$. Again, it is easy to see that $\lam_1 = 1/2$ and $\lam_2 = 1/4$. Hence,
\begin{eqnarray*}
	\Theta_f & = & \left(\frac{b_1}{\lam_1}\right)^{\lam_1} \cdot \left(\frac{b_2}{\lam_2}\right)^{\lam_2} \ = \ (2 \cdot 2)^{1/2} \cdot (1 \cdot 4)^{1/4} \ = \ 2 \cdot \sqrt{2} \ \approx \ 2.828.
\end{eqnarray*}

And since $|c| < \Theta_f$, we can conclude that $f$ is a strictly positive polynomial.

\subsection{A-Discriminants and Gale Duals}
\label{SubSec:GaleDual}

For a given $(n+1) \times m$ support matrix $M^A$ with $A \subset \Z^n$ and $\conv(A)$ being full dimensional, a \textit{Gale dual} or \textit{Gale transformation} is an integral $m \times (m-n-1)$ matrix $M^B$ such that its rows span the $\Z$-kernel of $M^A$. In other words, for every integral vector $\mathbf{v} \in \Z^m$ with $M^A \mathbf{v} = 0$, it holds that $\mathbf{v}$ is an integral linear combination of the rows of $M^B$, see \cite{GKZ:discriminant,Passare:Tsikh:Survey}.

If $A$ is a circuit, then $M^B$ is a vector with $n+2$ entries. It turns out that this vector is closely related to the global minimum $e^{\mathbf{s}^*} \in \R^n$ and the circuit number $\Theta_f$.

\begin{cor}
Let $f = \sum_{j = 0}^n b_j \mathbf{x}^{\alp(j)} + c \mathbf{x}^y$ be a polynomial supported on a circuit $A$ of the Form \eqref{Equ:OurPolynomials}.
Let $e^{\mathbf{s}^*} \in \R^n$ denote the global minimizer and $\Theta_f$ the circuit number. Then the Gale dual $M^B$ of the support matrix $M^A$ is an integral multiple of the vector
\begin{eqnarray*}
	\left(b_0 e^{\langle \mathbf{s}^*,\alp(0) \rangle},\ldots,b_n e^{\langle \mathbf{s}^*,\alp(n) \rangle},-\Theta_f e^{\langle \mathbf{s}^*,y \rangle}\right) \ \in \ \R^{n+2}.
\end{eqnarray*}
\label{Cor:GaleDual}
\end{cor}

\begin{proof}
The Gale dual $M^B$ needs to satisfy $M^A (M^B)^t = 0$. Since $A$ is a circuit, $M^B$ spans a one dimensional vector space. From $y = \sum_{j = 0}^n \lam_j \alp(j)$ it follows by construction of $e^{\mathbf{s}^*}$ and $\Theta_f$ (see proof of Proposition \ref{Prop:GlobalMinimizer}) that
\begin{eqnarray*}
	\left(b_0 e^{\langle \mathbf{s}^*,\alp(0) \rangle},\ldots,b_n e^{\langle \mathbf{s}^*,\alp(n) \rangle},-\Theta_f e^{\langle \mathbf{s}^*,y \rangle}\right) & = & (\lam_0,\ldots,\lam_n,-1)
\end{eqnarray*}
and the statement follows by definition of $M^A$ and $y$.
\end{proof}

Furthermore, we point out that the circuit number $\Theta_f$ and the question of nonnegativity is closely related to $A$-discriminants. Let $A = \{\alp(1),\ldots,\alp(d)\} \subset \Z^n$ and let $\C^A$ denote the space of all polynomials $\sum_{j = 1}^d b_j \mathbf{z}^{\alp(j)}$ with $b_j \in \C$. Since every (Laurent-) polynomial in $\C^A$ is uniquely determined by its coefficients, $\C^A$ can be identified with a $\C^d$ space. Let $\nabla_A$ be the Zariski closure of the subset of all polynomials $f$ in $\C^A$ for which there exists a point $\mathbf{z} \in (\C^*)^n$ such that
\begin{eqnarray*}
	f(\mathbf{z}) \ = \ 0 \ \text{ and } \ \frac{\partial f}{\partial z_j}(\mathbf{z}) \ = \ 0 \ \text{ for all } 1 \leq j \leq n.
\end{eqnarray*}
It is well known that $\nabla_A$ is an irreducible $\Q$-variety. If $\nabla_A$ is of codimension 1, then the $A$-\textit{discriminant} $\Delta_A$ is the integral, irreducible monic polynomial in $\C[b_1,\ldots,b_d]$, which has the variety $\nabla_A$, see \cite{GKZ:discriminant}. 

The following statement is an immediate consequence of Proposition \ref{Prop:GlobalMinimizer} and Theorem \ref{Thm:Positiv}. But it was (at least implicitly) already known before and can also be derived from \cite{GKZ:discriminant}, \cite{TTdW:genusone}, and \cite{Nie:Discriminants}.

\begin{cor}
The $A$-discriminant vanishes at a polynomial $f \in P_{\Delta}^y$ if and only if $f \in \partial P_{n,2d}^y$ or, equivalently, if and only if $c \in \{\pm \Theta_f\}$ and $y \notin (2\N)^n$ or $c = -\Theta_f$ and $y \in (2\N)^n$.
\label{Cor:ADiscriminant}
\end{cor}

\section{Amoebas of Real Polynomials Supported on a Circuit}
\label{Sec:amoebas}

In this section, we investigate amoebas of real polynomials supported on a circuit. We show that for amoebas of polynomials of the Form \eqref{Equ:OurPolynomials}, which are not a sum of monomial squares, a point $\mathbf{w}$ is contained in a bounded component of the complement only if the norm of the ``inner'' monomial is greater than the sum of all ``outer'' monomials at $\mathbf{w} \in \R^n$ (Theorem \ref{Thm:AmoebaSolidness}). This implies particularly that an amoeba of this type has a bounded component in the complement if and only if the ``inner'' coefficient $c$ satisfies $|c| > |\Theta_f|$, which proves the equivalence of (1) and (2) in Theorem \ref{Thm:MainEquivalences}. Furthermore, this result generalizes some statements in \cite{TTdW:genusone}.\\

In this section, we always assume that $f_c$ is a parametric family of a Laurent polynomial of the Form \eqref{Equ:OurPolynomials} with real parameter $c \in \R_{\leq 0}$. Furthermore, we always assume that $f_c$ is given in zero standard form (see Section \ref{Sec:Nonnegative}), i.e.,
\begin{eqnarray}
	f_c & = & \sum_{j = 1}^{n+1} b_j \mathbf{x}^{\alp(j)} + c, \label{Equ:OurPolynomialZeroStandard}
\end{eqnarray}
with $b_1,\ldots,b_{n+1} \in \R_{> 0}$. Let $\mathbf{w} \in \R^n$ be an arbitrary point in the underlying space
of $\cA(f_c)$. As introduced in Section \ref{SubSec:PreliminariesAmoebas}, we denote the fiber with respect to the $\Log|\cdot|$-map as $\F_{\mathbf{w}}$ and the fiber function of $f_c$ at the fiber $\F_{\mathbf{w}}$
as $f_c^{|\exp(\mathbf{w})|}$. We define the following parameters:
\begin{eqnarray*}
	\Theta_{\mathbf{w}} & = & \sum_{j = 1}^{n+1} |b_j e^{\langle \mathbf{w}, \alp(j) \rangle}|, \\
	\Psi_{\mathbf{w}} & = & \max_{1 \leq j \leq n+1} |b_j e^{\langle \mathbf{w}, \alp(j) \rangle}|.
\end{eqnarray*}

The following facts about amoebas supported on a circuit are well-known.

\begin{thm}[Purbhoo, Rullg{\aa}rd, Theobald, de Wolff]
Let $f = \lam_0 + \sum_{j = 1}^n b_j \mathbf{z}^{\alp(j)} + c \mathbf{z}^{y} \in \C[z_1^{\pm 1},\ldots,z_n^{\pm 1}]$ be a Laurent polynomial with $b_j \in \C^*$ and $c \in \C$ such that $\New(f)$ is a simplex and $y \in \Int(\New(f))$.
\begin{enumerate}
	\item The complement of $\cA(f)$ has exactly $n+1$ unbounded and at most one bounded component. If the bounded component $E_y(f)$ exists, then it has order $y$.
	\item $\mathbf{w} \in E_y(f) \subset \R^n$ only if $|c| > \Psi_{\mathbf{w}}$.
	\item $\mathbf{w} \in E_y(f) \subset \R^n$ if $|c| > \Theta_{\mathbf{w}}$.
	\item The complement of $\cA(f)$ has a bounded component if $|c| > \Theta_f$ and the bound is sharp if there exists a point
$\mathbf{\phi}$ on the unit torus $(S^1)^n \subset (\C^*)^n$ such that the fiber function $f^{1}$ satisfies $f^{1}(\mathbf{\phi}) = e^{i \psi} \cdot (\sum_{j = 0}^n |b_{\alp(j)}| - |c|)$ for some $\psi \in [0,2\pi)$.
\end{enumerate}
\label{Thm:MainGenus1}
\end{thm}

Part (1) and (2) are consequences of a Theorem by Rullg{\aa}rd based on tropical geometry, which was applied to the circuit case by Theobald and the second author, see \cite[Lemma 2.1]{TTdW:genusone} and also \cite[Theorem 4.1]{deWolff:Diss}. Part (3) is an immediate consequence of Purbhoo's lopsidedness condition (also referred as generalized Pellet's Theorem), see \cite{Purbhoo}. Part (4) is \cite[Theorem 4.4]{TTdW:genusone} after investigating $f$ in the standard form introduced in Section \ref{Sec:Nonnegative}, which guarantees that the bound given in \cite[Theorem 4.4]{TTdW:genusone} coincides with the circuit number $\Theta_f$. Note that this means $\Theta_f = \min_{\mathbf{w} \in \R^n} \Theta_{\mathbf{w}}$. Similarly, we define $\Psi_{f} = \min_{\mathbf{w} \in \R^n} \Psi_{\mathbf{w}}$. We remark that $\Psi_f$ is the minimal choice for $|c|$ such that the tropical hypersurface $\cT(\trop(f))$ of the tropical polynomial $\trop(f) = \bigoplus_{j = 1}^{n+1} \log|b_j| \odot \mathbf{x}^{\alp(j)} \oplus \log|c|$ has genus one, see \cite{deWolff:Diss,TTdW:genusone} for details; for an introduction to tropical geometry see \cite{MacLagan:Sturmfels}.\\

Summarized, Theorem \ref{Thm:MainGenus1} yields that the complement of an amoeba $\cA(f)$ of a real polynomial $f \in P_{\Delta}^y$ has a bounded component for all choices of $c < -\Theta_f$, the complement of $\cA(f)$ has no bounded component for $c \in [-\Psi_{f},0]$, and the situation is unclear for $c \in (-\Theta_f,-\Psi_f)$, see Figure \ref{Fig:Genus}. Hence, our goal in this section is to show the following theorem. 
\begin{thm}
Let $f_c$ be of the Form \eqref{Equ:OurPolynomialZeroStandard} such that $b_1,\ldots,b_{n+1} \in \R_{> 0}$ and $\mathbf{w} \in \R^n$. Then $\mathbf{w} \in \cA(f_c)$ for every real $c \in [-\Theta_{\mathbf{w}},-\Psi_{\mathbf{w}}]$.
\label{Thm:AmoebaSolidness}
\end{thm}

\begin{figure}
\ifpictures
$$
\includegraphics[width=0.4\linewidth]{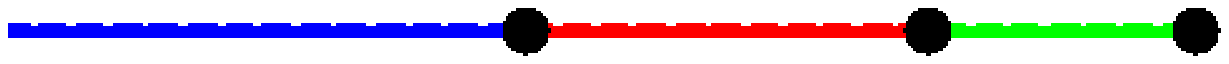}
\put(-7,0){\small{\mbox{$0$}}}
\put(-58,0){\small{\mbox{$-\Psi_f$}}}
\put(-118,0){\small{\mbox{$-\Theta_f$}}}
$$
\fi
\caption{Existence of a bounded component in the complement in dependence of the choice of the ``inner'' coefficient. If $c$ is contained in the left (blue) interval, then the complement of $\cA(f_c)$ has a bounded component. If $c$ is contained in the right (green) interval, then $\cA(f_c)$ is solid. But if $c$ is contained in the middle (red) interval, then it is in general unclear, whether the complement of $\cA(f_c)$ has a bounded component or not.}

\label{Fig:Genus}
\end{figure}

Note that for real polynomials $f_c \in P_{\Delta}^y$ we have $\cA(f_c) = \cA(f_{-c})$ if and only if $y \notin (2\N)^n$, since, if $y_j$ is odd and some $\textbf{w} \notin \cA(f_c)$, then $f_c(\mathbf{z}) \neq 0$ for all $\mathbf{z}$ contained in the fiber torus $\F_{\mathbf{w}} = \{\mathbf{z} : \Log|\mathbf{z}| = \mathbf{w}\}$. On the one hand, this torus is invariant under the variable transformation $z_j \mapsto -z_j$. On the other hand, this transformation transforms $f_c$ to $f_{-c}$. Therefore, Theorem \ref{Thm:AmoebaSolidness} implies particularly the following corollary, which is literally the equivalence between Part (1) and (2) in our main Theorem \ref{Thm:MainEquivalences}.   

\begin{cor}
Let $f_c$ be a polynomial in $P_{\Delta}^y$ such that $f$ is not a sum of monomial squares. Then $\cA(f_c)$ is solid if and only if $|c| \in [0,\Theta_f]$.
\label{Cor:AmoebaSolidness}
\end{cor}

\begin{proof}
The corollary follows immediately from Theorem \ref{Thm:MainGenus1} (2) and (4), Theorem \ref{Thm:AmoebaSolidness} (including its consecutive note) and the fact that $\Theta_f = \Theta_{\mathbf{s}^*} = \min_{\mathbf{w} \in \R^n} \Theta_{\mathbf{w}}$ by Corollary \ref{Cor:GlobalMinimizer}.
\end{proof}

The proof of Theorem \ref{Thm:AmoebaSolidness} will be quite a lot of work. We need to show a couple of technical statements before we can tackle the actual proof. The first lemma which we need was similarly used in \cite[Theorem 4.1]{TTdW:genusone}.

\begin{lemma}
Let $g: S^1 \to \C, \phi \mapsto b_1 e^{i r \phi} + b_2 e^{i \cdot (\eta + s \phi)}$ for some $b_1, b_2 \in \C^*$ with $|b_1| \geq |b_2|$, $\eta \in [0,2\pi)$ and $r,s \in \N^*$. Then there exist some $\phi, \phi' \in [0,2\pi)$ such that $g(\phi) \in \R_{\geq 0}$ and $g(\phi') \in \R_{\leq 0}$.
\label{Lem:Rouche}
\end{lemma}

For convenience, we provide the proof again. It is mainly based on the Rouch\'{e} theorem. Recall that the winding number of a closed curve $\gamma$ in the complex plane around a point $z$ is given by $\frac{1}{2\pi i} \int_{\gamma} \frac{d \zeta}{\zeta - z}$.

\begin{proof}
Assume $|b_1| > |b_2|$. Clearly, the function $b_1 \cdot e^{i \cdot r \phi}$ has a non-zero winding number around the origin. If $g$ would have a winding number of zero around the origin, then there would exist some $t \in (0,1)$ such that $h(\phi) = b_1 \cdot e^{i \cdot r \phi} + t \cdot b_2 \cdot e^{i \cdot (\eta + s \phi)}$ has a zero $\phi$ outside the origin. This is a contradiction. Hence, the trajectory of $g$ needs to intersect the real line in the strict positive part as well as in the strict negative part.

Since $g$ is continuous in the norms of its coefficients, the statement can be extended to $|b_1| = |b_2|$ and intersections of $g$ with the nonnegative part as well as the nonpositive part of the real axis.
\end{proof}

Now, we step over to complex functions on the real $n$-torus $(S^1)^n$.

\begin{lemma}
Let $g: (S^1)^n \to \C, \phi \mapsto \sum_{j = 1}^n b_j \cdot e^{i \phi_j} + b_{n+1} \cdot e^{-i \sum_{j = 1}^n \lam_j \phi_j}$ with $b_1 \geq \ldots \geq b_{n+1} \in \R_{> 0}$ and $\lam_j \in \Q$. There exists a path $\gamma: [0,1] \to (S^1)^n$ such that $g(\gamma) \in \R$, $g(\gamma(0)) = \sum_{j = 0}^n b_j$ and $g(\gamma(1)) \leq b_1 + b_n + b_{n+1} - \sum_{j = 2}^{n-1} b_j$.
\label{Lem:DimensionTwoSuffices}
\end{lemma}

\begin{proof}
We set $\gamma(0) = \mathbf{0} \in (S^1)^n$. We construct $\gamma: [0,1] \to (S^1)^n$ piecewise on intervals $[k_{j-1},k_j] \subset [0,1]$ for every $j \in \{2,\ldots,n-1\}$ with $k_1 = 0$ and $k_{n-1} = 1$. In every interval $[k_{j-1},k_j]$ we only vary $\phi_1,\phi_j$ and $\phi_n$ and leave all other $\phi_r$ invariant. I.e., in every interval $[k_{j-1},k_j]$ we only change the first, $j$-th, $n$-th and $(n+1)$-st term.  

In the interval $[k_{j-1},k_j]$, we continuously increase $\phi_j$ from $0$ to $\pi$. For every $\phi_j \in [0,\pi]$ there exists $\phi_1 \in [-\pi/2,0]$ such that $\IM(b_1 e^{i \phi_1} + b_j e^{i \phi_j}) = 0$, since $|b_1| \geq |b_j|$. For every pair $(\phi_1,\phi_j) \in [-\pi/2,0] \times [0,\pi]$, we find, by Lemma \ref{Lem:Rouche}, a $\phi_n$ such that $\IM(b_n e^{i \phi_n} + b_{n+1} e^{-i \sum_{j = 1}^n \lam_j \phi_j}) = 0$ by setting $\eta = -\sum_{j = 1}^{n-1} \lam_j \phi_j$ in Lemma \ref{Lem:Rouche}. Hence, for every $l \in [k_j,k_{j+1}]$ we have $g(\gamma(l)) \in \R$. And since $g$ is a smooth function, we obtain a smooth path segment in $(S^1)^n$ with smooth real image under $g$.

At the endpoint $\gamma(k_j)$ of the path segment $[k_{j-1},k_j] \subset [0,1]$, we are therefore in the situation $g(\gamma(k_j)) \leq |b_1| + |b_{j+1}| + \cdots + |b_{n-1}| + \RE(b_n) + \RE(b_{n+1}) - \sum_{l = 2}^{j} |b_l|$. We can glue together different path segments, since for each $\gamma(k_j)$ we have $\phi_1 = 0$ by construction and the value of $\phi_n$ does not matter. Thus, we can subsequently repeat the procedure for all $j$ until we reach $j = n-1$ and obtain a complete path $\gamma \subset (S^1)^n$ with the desired properties.
\end{proof}

For the next step of the proof we need to recall the definition of a hypotrochoid. A \textit{hypotrochoid} with parameters $R,r \in \Q_{> 0}$, $d \in \R_{> 0}$ satisfying $R \geq r$ is the plane algebraic curve $\gamma$ in $\R^2 \cong \C$ given by
\begin{eqnarray}
	\gamma: [0,2\pi) \ra \C, \quad \phi \mapsto (R - r) \cdot e^{i \cdot \phi} + d \cdot e^{i \cdot \lf(\frac{r - R}{r}\ri) \cdot \phi}.
	\label{Equ:Hypotrochoid}
\end{eqnarray}

Geometrically, a hypotrochoid is given the following way: Let a small circle $C_1$ with radius $r$ roll along the interior of a larger circle $C_2$ with radius $R$. Mark a point $p$ at the end of a segment with length $d$ starting at the center of $C_1$. Then the hypotrochoid is the trajectory of $p$.

We say that a curve $\gamma$ is a \textit{hypotrochoid up to a rotation} if there exists some re-parametrization $\rho_k : [0,2\pi) \ra [0,2\pi), \phi \mapsto k + \phi \mod 2 \pi$ with $k \in
[0,2\pi)$ such that $\gamma \circ \rho_k^{-1}$ is a hypotrochoid. If $k = 0$ or $k = \pi$, then we say that $\gamma$ is a \textit{real hypotrochoid}. Hypotrochoids are closed, continuous curves in the complex plane, which attain values in the closed annulus with outer radius $(R-r) + d$ and inner radius $(R-r) - d$ for $(R - r) \geq d$. Furthermore, if they are real, then they are symmetric along the real line. For an overview about hypocycloids and other plane algebraic curves see \cite{Brieskorn:Knoerrer}.

In order to prove the second key lemma, which is needed for the proof of Theorem \ref{Thm:AmoebaSolidness}, we make use of the following special case of \cite[Theorem 4.1]{Theobald:deWolff:Trinomials}.

\begin{lemma}
Let $g: S^1 \to \C, \phi \mapsto e^{i s \phi} + q e^{-i t \phi} + p$ with $p,q \in \C^*$. Then $g$ is a hypotrochoid up to a rotation around the point $p$ with parameters $R = (t + s) / t$, $r = s / t$ and $d = |q|$ rotated by $\arg(q) \cdot s$.
\label{Lem:Hypotrochoid}
\end{lemma}

The proof of this lemma is a straightforward computation.

\begin{proof}
The non-constant part $g - p$ of the function $g$ is given by
\begin{eqnarray*}
	(g-p)(\phi) & = & e^{i s \phi} + |q| \cdot e^{i \cdot (\arg(q) - t \phi)}.
\end{eqnarray*}
Since, with our choice of parameters, $R - r = 1$ and $(r - R)/r = -t/s$, it follows by \eqref{Equ:Hypotrochoid} after replacing $\phi$ by $\phi' = s \phi$  that $g - p$ is a hypotrochoid up to a rotation.
\end{proof}

Lemma \ref{Lem:Hypotrochoid} about hypotrochoids allows us to prove the following technical lemma.

\begin{lemma}
Let $g: (S^1)^2 \to \C, (\phi_1,\phi_2) \mapsto b_1 e^{i \phi_1} + b_2 e^{i \phi_2} + b_{3} e^{-i \cdot (\lam_1 \phi_1 + \lam_2 \phi_2)}$ with $b_1, b_2 \in \R_{> 0}$, $b_3 \in \R^*$, $|b_1| \geq |b_2| \geq |b_3|$ and $\lam_1,\lam_2 \in \Q$. Then $g$ attains all real values in the interval $ [b_1,b_1 + b_2 + b_3] \subset \R_{> 0}$.
\label{Lem:DimensionTwo}
\end{lemma}

With this Lemma we have everything what is needed to prove Theorem \ref{Thm:AmoebaSolidness}. We provide the quite long and technical proof of Lemma \ref{Lem:DimensionTwo} after the proof of Theorem \ref{Thm:AmoebaSolidness}.

\begin{proof}{(Proof of Theorem \ref{Thm:AmoebaSolidness})}
Let $f_c = \sum_{j = 1}^{n+1} b_j \mathbf{x}^{\alp(j)} - c$ with $b_1 \geq \cdots \geq b_{n+1} \in \R_{> 0}$ and $\mathbf{w} \in \R^n$ such that $\{\alp(1),\ldots,\alp(n+1),0\} \subset \Z^n$ forms a circuit with $0$ in the interior of $\conv \{\alp(1),\ldots,\alp(n+1)\}$. By Corollary \ref{Cor:StandardZeroForm}, we can assume that $f_c$ is in zero standard form, i.e., we can assume that $\alp(j) = \mu e_j \in \N^n$ for $1 \leq j \leq n$, with $\lam_j \in \Q^*$, $\mu \in \N^*$ denoting the least common multiple of the denominators of $\lam_0,\ldots,\lam_n \in \Q_{|(0,1)}$ and $e_j$ denoting the standard basis vector. By construction, we have $\alp(n+1) = - \mu / \lam_0 \cdot \sum_{j = 1}^n \lam_j e_j \in \Z^n$. Furthermore, we can assume without loss of generality $\mathbf{w} = \mathbf{1}$ after adjusting the coefficients $b_j$ if necessary. \\

\noindent We investigate the fiber function
\begin{eqnarray*}
	f_c^{|\mathbf{1}|} & = & \sum_{j = 1}^n b_j e^{i \mu \phi_j} + b_{n+1} e^{-i \cdot \frac{\mu}{\lam_0} \cdot \sum_{j = 1}^n \lam_j \phi_j} - c.
\end{eqnarray*}
We have to show that $\cV(f_c^{|\mathbf{1}|}) \neq \emptyset$ for all $c$ with $\Psi_f = \Psi_{\mathbf{1}} = |b_1| \leq c \leq \Theta_{\mathbf{1}}$. By applying Lemma \ref{Lem:DimensionTwoSuffices}, $f_c^{|\mathbf{1}|}$ attains all real values in the interval $[|b_1| + |b_n| + |b_{n+1}| - \sum_{j = 2}^{n-1} |b_j| - c, \Theta_\mathbf{1} - c]$. Hence, if $|b_n| + |b_{n+1}| - \sum_{j = 2}^{n-1} |b_j| \leq 0$, then we are done. This is always the case, if $n \geq 4$ or if $n = 3$ and $b_2 \geq b_3 + b_4$.\\

Let now $n \in \{2,3\}$. If $n = 3$, the we apply Lemma \ref{Lem:DimensionTwoSuffices}, fix $\phi_2 = \pi$, and restrict $f_c^{|\mathbf{1}|}$ to
\begin{eqnarray*}
	g(\phi_1,\phi_n)_c & = & b_1 e^{i \phi_1} + b_n e^{i \phi_n} + b_{n+1} e^{- i \cdot \frac{\mu}{\lam_0} \cdot (\pi \lam_2 + \lam_1 \phi_1 + \lam_n \phi_n)} - \sum_{1 < j < n} b_j - c,
\end{eqnarray*}
which is defined on the sub 2-torus of $\F_{|\mathbf{1}|}$ given by $(\phi_1,\phi_n)$. Since $(\mu \cdot \lam_2) / \lam_0$ is an integer, $b_{n+1} \cdot e^{-i \cdot \mu \cdot \lam_2 / \lam_0}$ is real and hence we can apply Lemma \ref{Lem:DimensionTwo}. It yields that $g(\phi_1,\phi_n)_c$ attains all real values in the interval $[b_1 - \sum_{1 < j < n} b_j - c, b_1 + b_n + b_{n+1} - \sum_{1 < j < n} b_j - c]$. Thus, all real values in the interval $[|b_1| + |b_n| + |b_{n+1}| - \sum_{j = 2}^{n-1} |b_j| - c, \Theta_\mathbf{1} - c]$ are attained by $f_c^{|\mathbf{1}|}$ and hence we find a root of $f_c^{|\mathbf{1}|}$ for every choice of $b_1 \leq c \leq \Theta_{\mathbf{1}}$. Therefore, $\mathbf{1} \in \cA(f_c)$ for all $c \in [-\Theta_f,-\Psi_f]$.
\end{proof}

We close the section with the proof of Lemma \ref{Lem:DimensionTwo}.

\begin{proof}{(Proof of Lemma \ref{Lem:DimensionTwo})}
For every fixed value of $\phi_1 \in [0,2\pi)$ the values of $g$ are given by a curve of the form
\begin{eqnarray*}
	h_{\phi_1}: [0,2\pi) \to \C, \quad \phi_2 \mapsto b_2 e^{i \phi_2} + b_3 e^{-i (\lam_1 \phi_1 + \lam_2 \phi_2)} + b_1 e^{i \phi_1}.
\end{eqnarray*}
By Lemma \ref{Lem:Hypotrochoid}, $h_{\phi_1}$ is a hypotrochoid up to a rotation around the point $b_1 e^{i \phi_1}$ attaining absolute values in the annulus $A_{\phi_1}$ with outer
radius $b_2 + b_3$ and inner radius $b_2 - b_3$ around the point $b_1 e^{i \phi_1}$ ($A_{\phi_1}$ degenerates to a disc for $b_2 = b_3$). Since $|b_2| \geq |b_3|$, it follows from Lemma \ref{Lem:Rouche} that $h_{0}$ intersects the real coordinate axis in both at least one point greater or equal than $b_1$ and at least one point less or equal than $b_1$. More specific, let $\phi_2(1),\ldots,\phi_2(k) \in [0,2\pi)$ denote the arguments such that $\mu_j = g(0,\phi_2(j)) \in \R$ with $\mu_1 \geq \cdots \geq \mu_k$. Analogously, we denote by $\phi_2(1)',\ldots,\phi_2'(l) \in [0,2\pi)$ the arguments such that $\nu_j = g(\pi,\phi_2'(j)) \in \R$ with $\nu_1 \geq \cdots \geq \nu_l$.
Note that $\mu_1 \geq b_1 + b_2 - b_3$, $\mu_k \leq b_1 - b_2 + b_3$ and $\nu_1 \leq -b_1 + b_2 + b_3$ and therefore
\begin{eqnarray*}
  \nu_1 \ \leq \ \mu_k \ \leq \ b_1 \ \leq \ b_1 + b_2 - b_3 \ \leq \ \mu_1.
\end{eqnarray*}

The key observation of the proof is the following: $h_{\phi_1}$ depends continuously on $\phi_1$. But this means that 
\begin{eqnarray*}
	& & H: [0,1] \times [0,2\pi) \ \to \ \C \\
	& & H(\phi_1/(2\pi),h_{\phi_1}(\phi_2)) \ = \ g(\phi_1,\phi_2) \ = \ b_1 e^{i \phi_1} + b_2 e^{i \phi_2} + b_3 e^{i (- \lam_1 \phi_1 - \lam_2 \phi_2)}
\end{eqnarray*}
is a homotopy of hypotrochoid curves along the circle with radius $b_1$. 

Since $[b_1,b_1 + b_2 + b_3] \subset [\mu_k,\mu_1] \subset \R$, the proof is completed, if we can show that all real values in the interval $[\mu_k,\mu_{k-1}] \cup \cdots \cup [\mu_2,\mu_{1}] = [\mu_k,\mu_1]$ are attained by $g$. 

Since $g(0,\phi_2)$ is a real hypotrochoid, i.e., in particular, connected and symmetric along the real line, 
for every $1 \leq j \leq k-1$ there exists a closed connected subset $\gamma_j$ of the trajectory of the hypotrochoid $g(0,\phi_2)$ and its pointwise complex conjugate $\ovl{\gamma_j}$ both connecting $\mu_j$ and $\mu_{j+1}$. I.e., $\rho_j = \gamma_j \cup \ovl{\gamma_j}$ forms a topological circle intersecting $\R$ exactly in $\mu_j$ and $\mu_{j+1}$ and thus its projection on $\R$ covers $[\mu_{j+1},\mu_{j}]$. Hence, $\bigcup_{j = 1}^{k-1} \rho_j$ projected on the real line covers $[\mu_k,\mu_1]$.

Now, we restrict the homotopy $H$ of hypotrochoids to a particular circle $\rho_j$ and to moving $\phi_1$ continuously from 0 to $\pi$, i.e., the induced homotopy is $H_j: \rho_j \times [0,1] \to \C$ of the circle $\rho_j$ moved around the half-circle $\{b_1 e^{i \cdot \phi_1} : \phi_1 \in [0,2\pi)\}$. Two cases can occur during the homotopy $H_j$: Either $\R$ intersects the circle $\rho_j$ transversally in two points during the whole homotopy, or there exists a point $\tau \in (0,1)$ such that the circle and $\R$ intersect non-transversally at $H_j(\rho_j,\tau)$.

First assume that there exists a point $\tau \in (0,1)$ along the homotopy such that $H_j(\rho_j,\tau)$ intersects the real line non-transversally in a single point $s \in \R$. Hence, $H_j$ yields in particular a new homotopy $\wh H_j: \{\mu_j,\mu_{j+1}\} \times [0,\tau] \to \R$ of both the two points $\mu_j$ and $\mu_{j+1}$ to $s$ along the real line. Thus, for all points $x \in [\mu_{j+1},\mu_j]$ there exists $\tau' \in [0,\tau]$ such that $x = \wh H_j(\mu_j,\tau')$ or $x = \wh H_j(\mu_{j+1},\tau')$, i.e., all points in $[\mu_{j+1},\mu_j]$ are visited during the homotopy $\wh H_j$ and hence every real value in $[\mu_{j+1},\mu_j]$ is attained by $g$ (see Figure \ref{Fig:AmoebaSolid1}). 

Now assume that $H_j(\rho_j,\tau)$ intersects the real line in two distinct points for every $\tau \in [0,1]$. Thus, again, there is an induced homotopy of points $\wh H_j: \{\mu_j,\mu_{j+1}\} \times [0,1] \to \R$ along the real line. Since $H_j$ is a restriction of $H$, we know that $\wh H_j(\{\mu_j,\mu_{j+1}\},1)$ are real points of the hypotrochoid $H(\pi,h_\pi(\phi_2))$, i.e., $\wh H_j(\{\mu_j,\mu_{j+1}\},1) \in \{\nu_1,\ldots,\nu_l\}$. Since $\nu_l \leq \cdots \leq \nu_1$ and  $\nu_1 \leq \mu_j$ for all $1 \leq j \leq k$, again, all points in $[\mu_{j+1},\mu_j]$ are visited during the homotopy $\wh H_j$ (see Figure \ref{Fig:AmoebaSolid2}).

\begin{figure}
\ifpictures
$$
\includegraphics[width=0.3\linewidth]{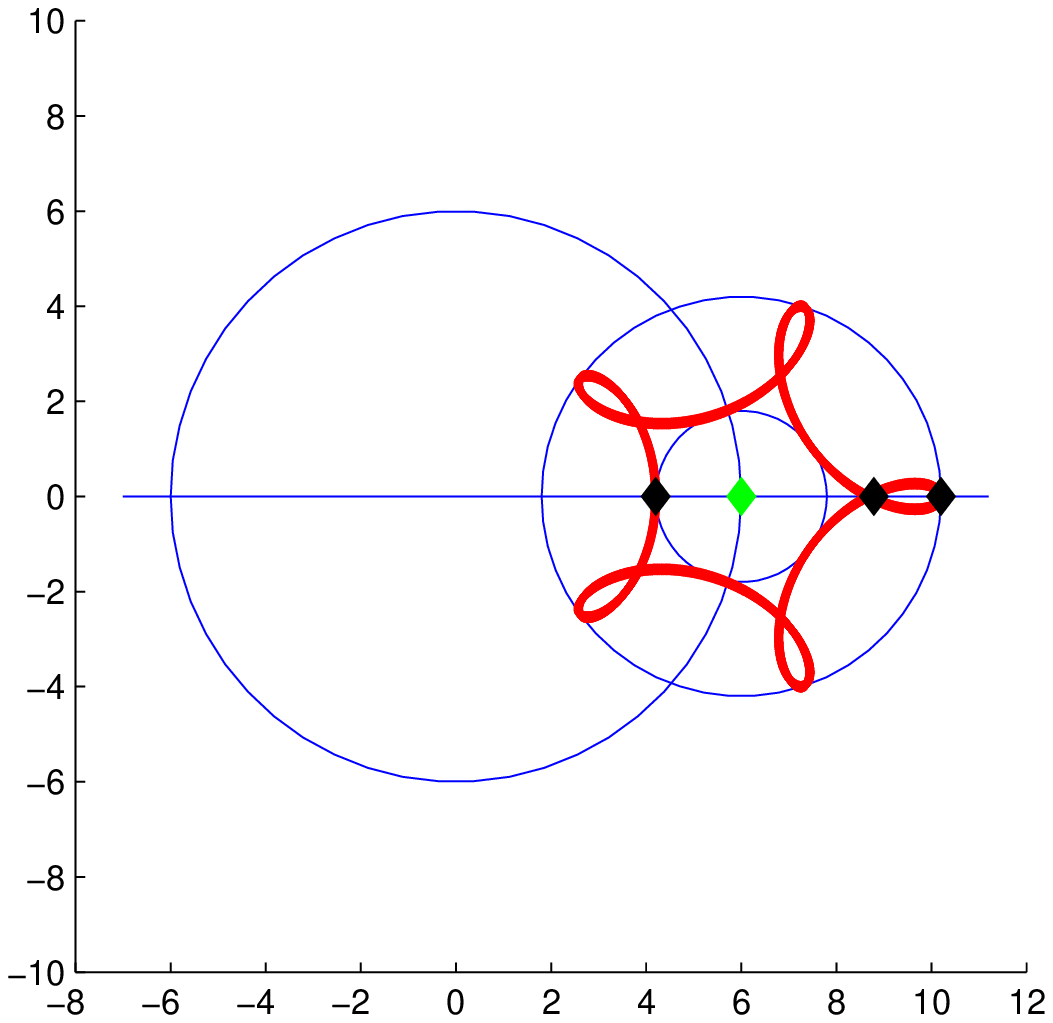}
\includegraphics[width=0.3\linewidth]{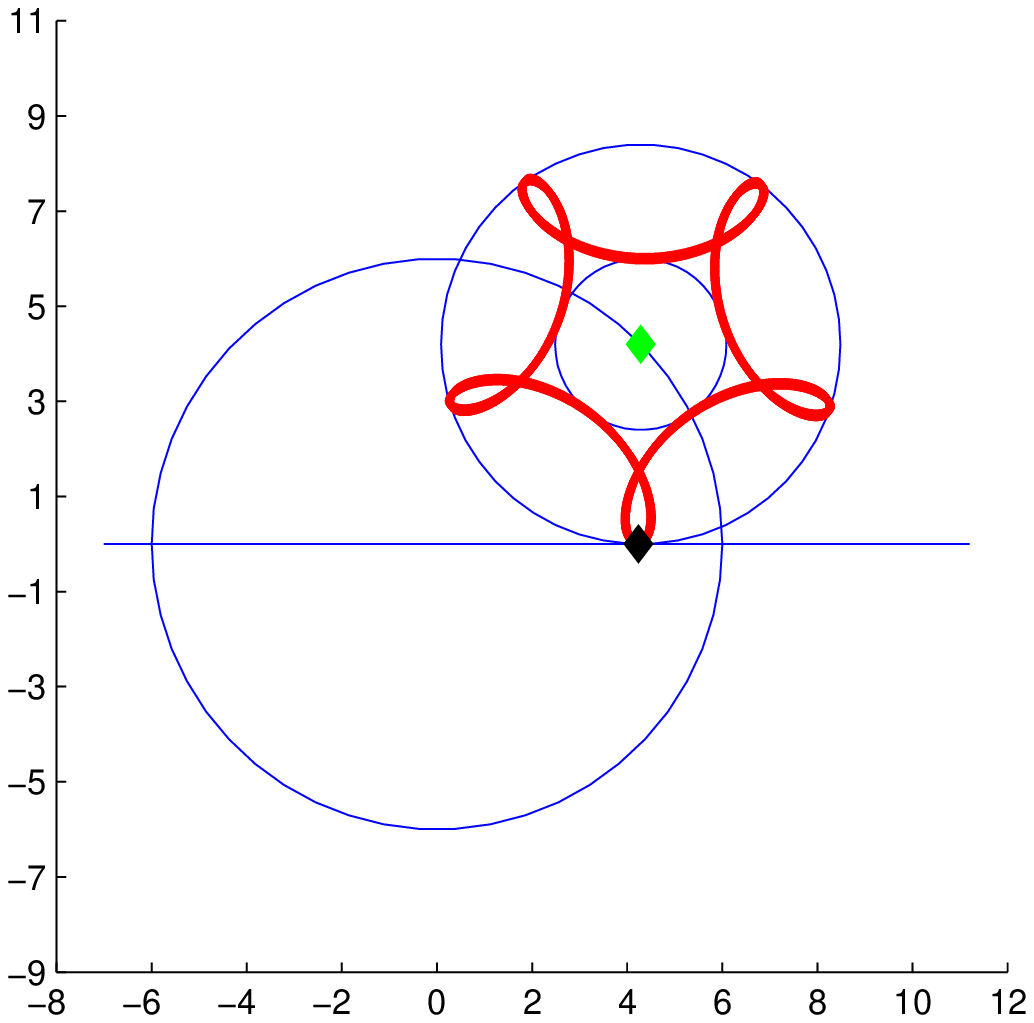}
\includegraphics[width=0.3\linewidth]{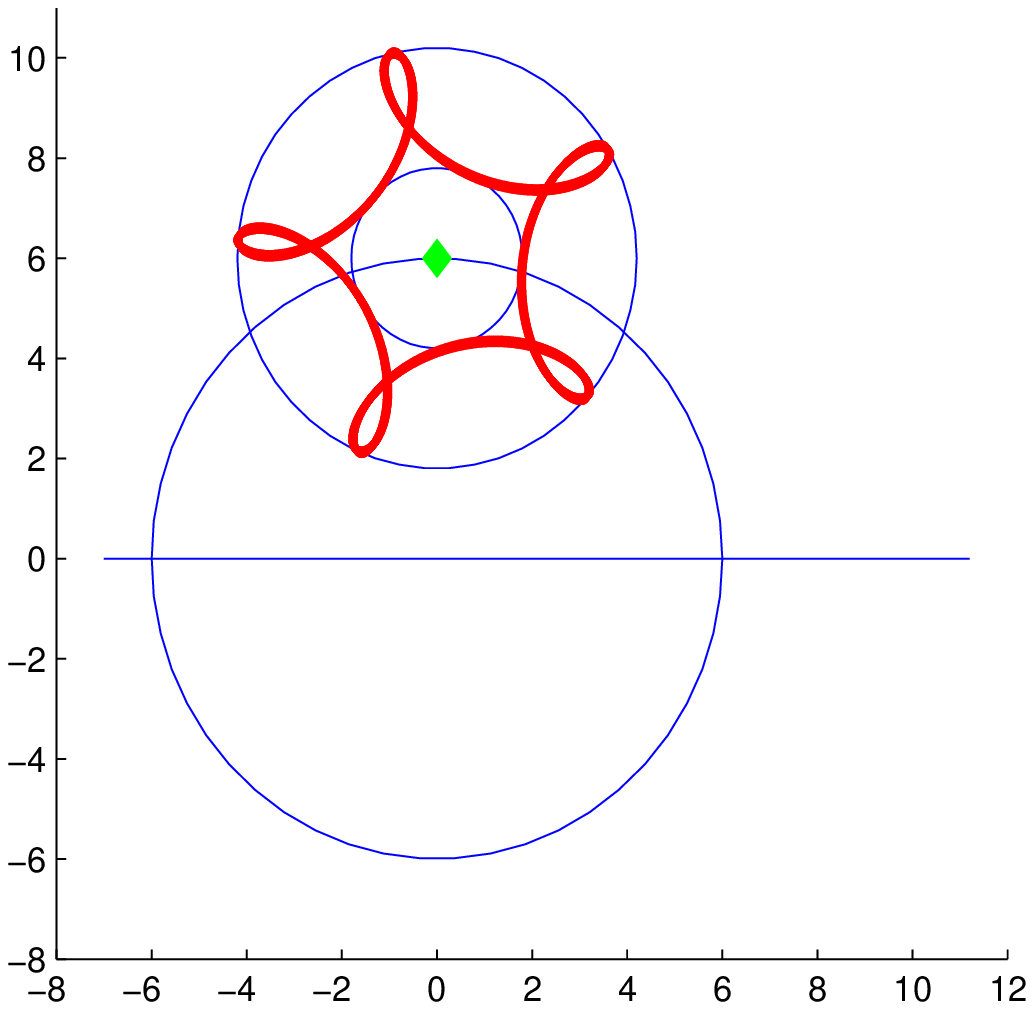}
$$
\fi
\caption{Homotopy of a hypotrochoid where the intersection of the hypotrochoid with the real line becomes empty during the homotopy.}
\label{Fig:AmoebaSolid1}
\end{figure}
\end{proof}

\begin{figure}
\ifpictures
$$
\includegraphics[width=0.3\linewidth]{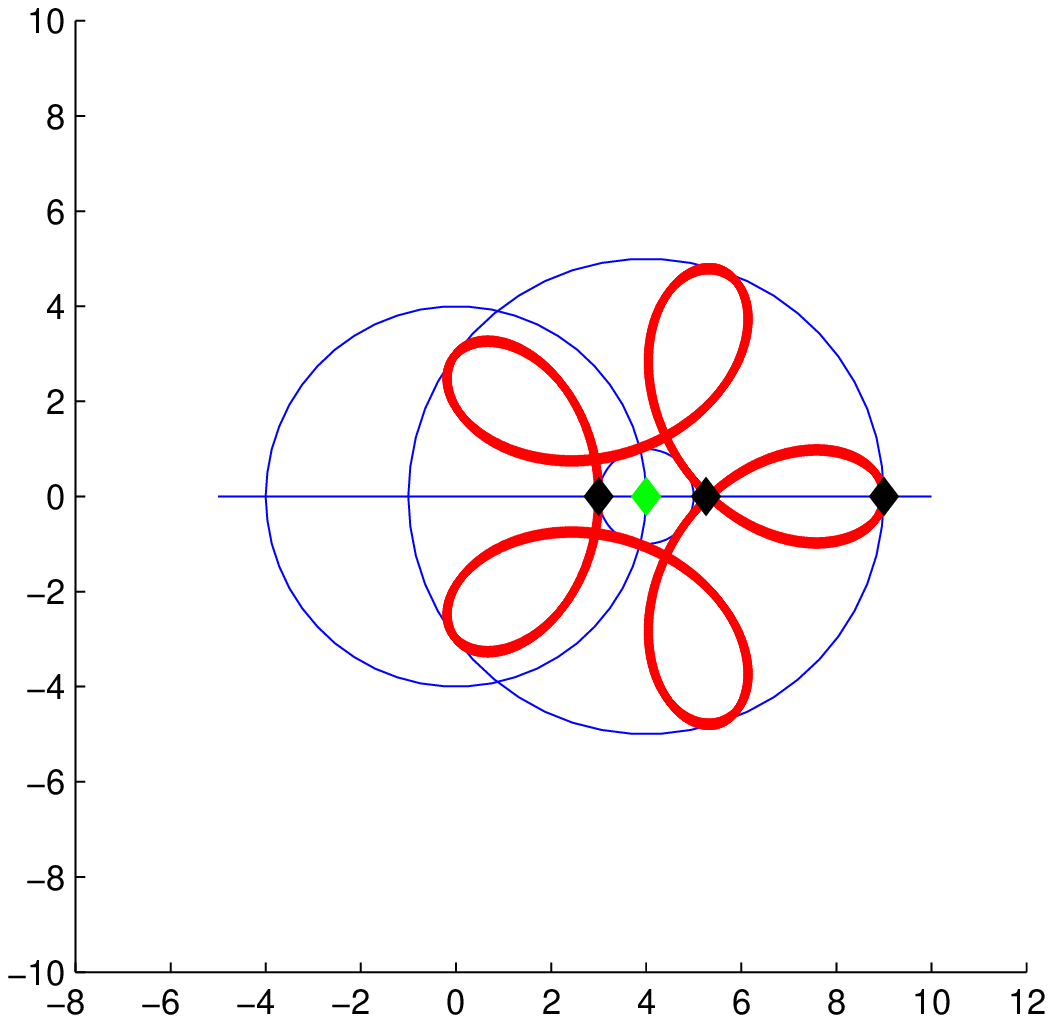}
\includegraphics[width=0.3\linewidth]{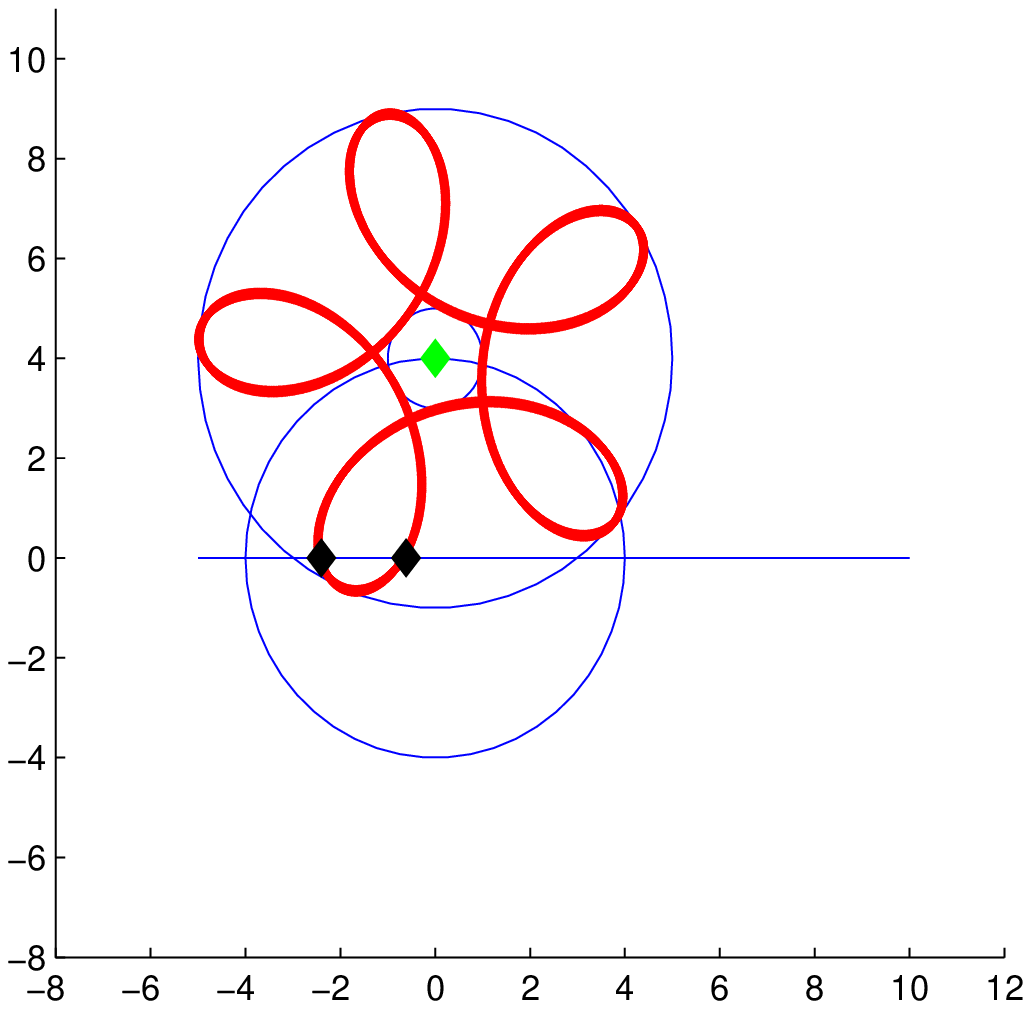}
\includegraphics[width=0.3\linewidth]{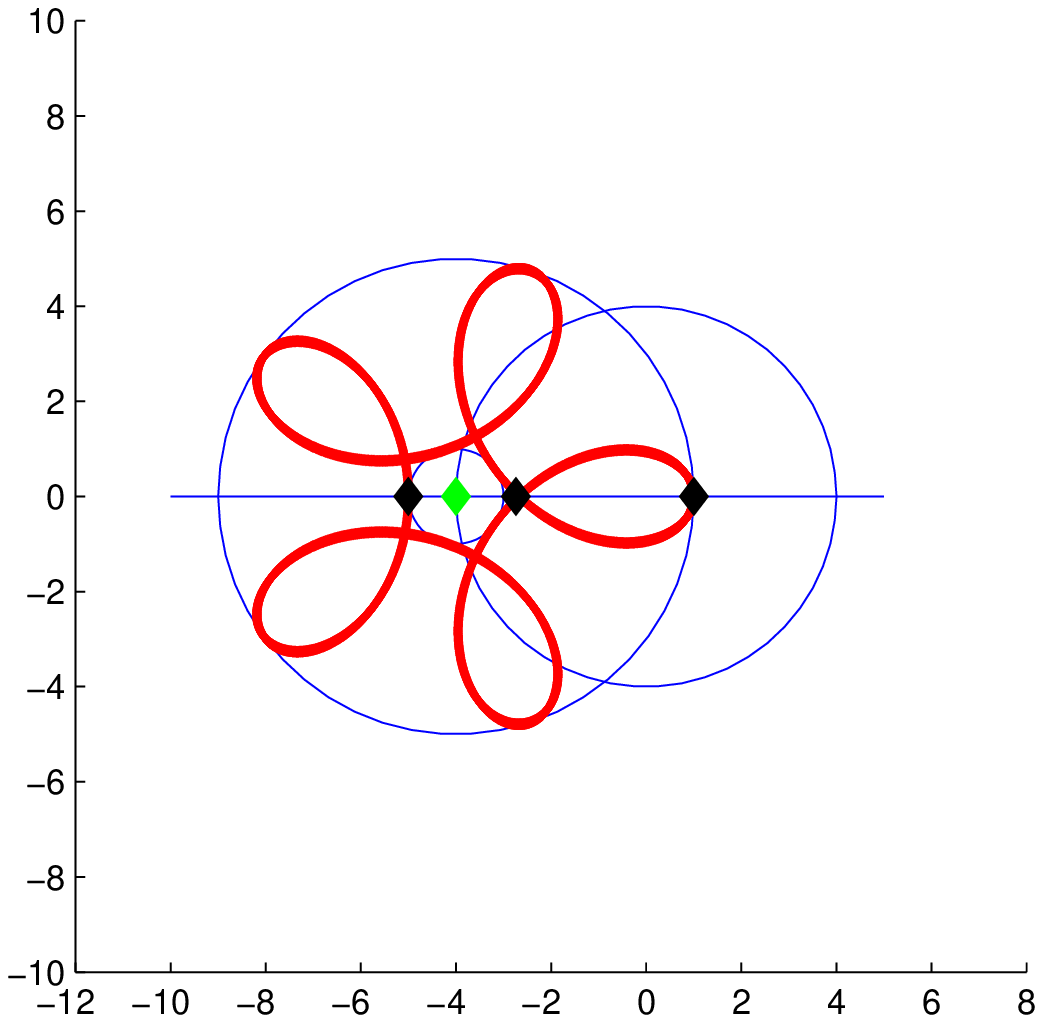}
$$
\fi
\caption{Homotopy of a hypotrochoid always intersecting the real line.}
\label{Fig:AmoebaSolid2}
\end{figure}

\section{Sums of Squares supported on a circuit}
\label{Sec:PSDSOS}

In this section we completely characterize the section $\Sigma_{n,2d}^y$. It is particularly interesting that this section depends heavily on the lattice point configuration in $\Delta$, thereby, yielding a connection to the theory of lattice polytopes and toric geometry. By investigating this connection in more detail, we will prove that the sections $P_{2,2d}^y$ and $\Sigma_{2,2d}^y$ almost always coincide and that $P_{n,2d}^y$ and $\Sigma_{n,2d}^y$ contain large sections, at which nonnegative polynomials are equal to sums of squares for $n > 2$, see Corollaries \ref{cor:R2} and \ref{cor:bound}.

Surprisingly, the sums of squares condition is exactly the same as for the corresponding agiforms. For this, we briefly review the Gram matrix method for sums of squares polynomials. For $d \in \N$ let $\N_d^n = \{\alpha\in\N^n : \alpha_1+\dots +\alpha_n \leq d\}$ and $p = \sum_{k=1}^{r}h_k^2$ where $p(\mathbf{x}) = \sum_{\alpha \in \N_{2d}^n}^{}a(\alpha)\mathbf{x}^{\alpha}$ and
$h_k(\mathbf{x}) = \sum_{\beta \in \N_d^n}^{}b_k(\beta)\mathbf{x}^{\beta}$. Let $B(\beta) = (b_1(\beta),\dots,b_r(\beta))$ and $G(\beta,\beta') = B(\beta)\cdot B(\beta') = \sum_{k = 1}^{r}b_k(\beta)b_k(\beta')$ with $\beta,\beta'\in\N_d^n$. Comparing coefficients one has
$$a(\alpha) \ = \ \sum_{\beta+\beta' = \alpha}^{}G(\beta,\beta') \ = \ \sum_{\beta \in \N^{n}_{d}} G(\beta,\alpha-\beta).$$ 
In this case, $[B(\beta)\cdot B(\beta')]_{\beta,\beta' \in \N_d^n}$ is a positive semidefinite matrix.

Furthermore, we need the following well-known lemma, see \cite{Blekherman:Parrilo:Thomas}.
\begin{lemma}
Let $f \in \Sig_{n,2d}$ be a sum of squares and $T \in GL_n(\R)$ be a matrix yielding a variable transformation $\mathbf{x} \mapsto T\mathbf{x}$. Then $f(T\mathbf{x})$ also is a sum of squares.
\label{Lem:SOSTranformationInvariant}
\end{lemma}

Now, we can characterize the sums of squares among nonnegative polynomials in $P_{\Delta}^y$.

\begin{thm}
\label{thm:sos}
 Let $f = \lambda_0 + \sum_{j=1}^n b_j \mathbf{x}^{\alpha(j)} + c\cdot \mathbf{x}^y \in P_{n,2d}^y$. Then
\begin{eqnarray*}
	f \in \Sigma_{n,2d}^y & \text{ if and only if } & y \in \Delta^* \ \text{ or } \ c > 0 \ \text{ and } \ y \in (2\N)^n.
\end{eqnarray*}
Furthermore, if $f \in \Sigma_{n,2d}^y$, then $f$ is a sum of binomial squares.
\end{thm}

Note again that for $f \in P_\Delta^y$ the condition $c > 0$ and $y \in (2\N)^n$ holds if and only if $f$ is a sum of monomial squares such that the above theorem holds trivially.

\begin{proof}
First, assume that $f \in \Sigma_{n,2d}^y$. We can assume that $c < 0$ by the following argument: If $y \in (2\N)^n$, then $f$ is obviously a sum of (monomial) squares for $c > 0$. If $y \notin (2\N)^n$ and $c > 0$, then, by Lemma \ref{Lem:SOSTranformationInvariant} and a suitable variable transformation as in the proof of Theorem \ref{Thm:Positiv}, we can reduce to the case $c < 0$. Let $f = \sum_{}^{}h_k^2$ and define $M = \{\beta : b_k(\beta) \neq 0 \,\,\, \textrm{for some}\,\,\, k\}$ with $\beta$ and $b_k(\beta)$ as in the Gram matrix method. Following \cite[Theorem 3.3]{Reznick:AGI}, we claim that the set 
$L = 2M \cup \hat\Delta \cup \{y\}$ is $\hat\Delta$-mediated and hence $y \in \Delta^*$. Here, $\hat\Delta$ is the set of vertices of $\Delta$. In order to show the claim we write every $\beta \in L\setminus \hat\Delta$ as a sum of two distinct points in $M$, which implies that $\beta$ is an average of two distinct points in $2M \subseteq L$. Note that if $G(\alpha,\alpha') < 0$, then $b_k(\alpha)b_k(\alpha') < 0$ for some $k$ and hence $\alpha\neq\alpha'$ and $\alpha, \alpha' \in M$. Hence, it suffices to show that for $\beta \in L\setminus\hat\Delta$ there exists an $\alpha$ with $G(\alpha,\beta-\alpha) < 0$. We have $a(y) = c < 0$, so $G(\alpha_0, y-\alpha_0) < 0$ for some $\alpha_0$. If $\beta \neq y$ then $\beta \in L\setminus(\hat\Delta \cup \{y\})$ and $a(\beta) = 0 = \sum_{}^{} G(\alpha,\beta-\alpha)$. But $\beta \in 2M$, so $G(\frac{1}{2}\beta,\frac{1}{2}\beta) > 0$ and hence there has to exist an $\alpha$ with $G(\alpha,\beta-\alpha) < 0$ to let the sum vanish.\\

Let now $y \in \Delta^*$. We investigate two cases. First, let $y \notin (2\mathbb N)^n$. Then it suffices to prove the statement for $c = \pm\Theta_f$ by the following argument: Let $f_1 =\lambda_0 + \sum_{j=1}^n b_j \mathbf{x}^{\alpha(j)} - c\cdot \mathbf{x}^y \in P_{n,2d}^y$ and
$f_2 = \lambda_0 + \sum_{j=1}^n b_j \mathbf{x}^{\alpha(j)} + c\cdot \mathbf{x}^y \in P_{n,2d}^y$. Let $c^*$ be such that $-c < c^* < c$ and $f_3 =  \lambda_0 + \sum_{j=1}^n b_j \mathbf{x}^{\alpha_j} + c^*\cdot \mathbf{x}^y \in P_{n,2d}^y$. Then we have $f_3 = \lambda_1f_1 + \lambda_2f_2$ with $\lambda_1 = \frac{c+c^*}{2c}$, $\lambda_2 = \frac{c-c^*}{2c}$ and $\lambda_1, \lambda_2 > 0$, $\lambda_1 + \lambda_2 = 1$. By the same argument involving the variable transformation $x_j \mapsto -x_j$ for some $j \in \{1,\ldots,n\}$ as before (proof of Theorem \ref{Thm:Positiv}, Lemma \ref{Lem:SOSTranformationInvariant}) it suffices to investigate the case $c = -\Theta_f$. 
 Consider the following linear transformation of the variables $x_1,\dots,x_n$.
$$T : (x_1,\dots,x_n) \mapsto \left((e^{s^*})_1 x_1,\dots,(e^{s^*})_nx_n\right),$$
where $(e^{s^*})_j$ denotes the $j$-th coordinate of the global minimizer $e^{\mathbf{s}^*}$ of $f$, see Proposition \ref{Prop:GlobalMinimizer} and proof of Theorem \ref{Thm:Positiv}. By Lemma \ref{Lem:SOSTranformationInvariant}, $f \in \Sigma_{n,2d}$ if and only if $f(T(\mathbf{x})) \in \Sigma_{n,2d}$, where
\begin{eqnarray}
	f(T(\mathbf{x})) & = & \lambda_0 + \sum_{j=1}^n \lambda_j\mathbf{x}^{\alpha(j)} - \mathbf{x}^y.\label{Equ:SOSAgiform}
\end{eqnarray}
But $f(T(\mathbf{x}))$ is the dehomogenization of an agiform and, therefore, by Theorem \ref{thm:rezmain}, $f \in \Sigma_{n,2d}^y$ if and only if $y \in \Delta^*$.\\

If $y \in (2\mathbb N)^n$, then we use the same argument to prove that $f$ is a sum of squares for $c = -\Theta_f$. For $c > -\Theta_f$, the polynomial $f$ is obviously a sum of squares, since the inner monomial can be written as $-\Theta_f \mathbf{x}^y$ plus the term $(c + \Theta_f)\mathbf{x}^y$, which is a square.

In \cite[Theorem 4.4]{Reznick:AGI} it is shown that the agiforms in \eqref{Equ:SOSAgiform} are sums of \emph{binomial} squares. Thus, for $y \in \Delta^*$, the nonnegative polynomials $f \in P_{n,2d}^y$ are also sums of binomial squares, since the binomial structure is preserved under the variable transformation $T$.
\end{proof}

Agiforms can be recovered by setting $b_j = \lambda_j$ and, hence, Theorems \ref{Thm:Positiv} and \ref{thm:sos} generalize results for agiforms in \cite{Reznick:AGI}. Furthermore, by setting $\alpha(j) = 2d \cdot e_j$ for $1\leq j\leq n$, we recover the dehomogenized version of what is called an \emph{elementary diagonal minus tail form} in \cite{Fidalgo:Kovacec}, and, again, Theorems \ref{Thm:Positiv} and \ref{thm:sos} generalize one of the main results in \cite{Fidalgo:Kovacec} to arbitrary simplices. 

We remark that in \cite{Reznick:AGI} an algorithm is given to compute such a sum of squares representation in the case of agiforms in Theorem \ref{thm:sos}, which can be generalized to arbitrary circuit polynomials. Furthermore, in \cite{Reznick:AGI} it is shown that every agiform in $\Sigma_{n,2d}^y$ can be written as a sum of $|L\setminus\hat\Delta|$ binomial squares. By using the variable transformation $T$ in the proof of Theorem \ref{thm:sos}, we conclude that a general circuit polynomial $f\in \Sigma_{n,2d}^y$ also can be written as a sum of $|L\setminus\hat\Delta| = |L| - (n + 1)$ binomial squares.\\

Theorem \ref{thm:sos} also comes with two immediate corollaries.

\begin{cor}
 Let $\Delta$ be an $H$-simplex and $f \in P_{\Delta}^y$. Then $f \in P_{n,2d}^y$ if and only if $f \in \Sigma_{n,2d}^y$.
\end{cor}

\begin{proof}
 Since $\Delta$ is an $H$-simplex, it holds that $\Delta^* = (\Delta\cap \Z^n)$ (see Section \ref{SubSec:Agiforms}) and we always have $y \in \Delta^*$. \ 
\end{proof}

The second corollary concerns sums of squares relaxations for minimizing polynomial functions. For this, note that the quantity $f_{sos}^* = \max\{\lambda : f - \lambda \in \Sigma_{n,2d}\}$ is a lower bound for $f^* = \min\{f(\mathbf{x}) : \mathbf{x} \in \R^n\}$, see for example \cite{Lasserre:Buch}.

\begin{cor}
\label{cor:relaxation}
 Let $f \in P_{\Delta}^y$. Then $f_{sos}^* = f^*$ if and only if  $y \in \Delta^*$.

\end{cor}

\begin{proof}
 We have $f_{sos}^* = f^*$ if and only if $f - f^* \in \Sigma_{n,2d}$. However, subtracting the minimum of the polynomial $f$ does not affect the question whether $y\in\Delta^*$ or not. Hence, if $y\in\Delta^*$, this will also hold for the nonnegative polynomial $f - f^*$ and vice versa.
\end{proof}

 As an extension, we consider in the following the case of multiple support points, which are interior lattice points in the simplex $\Delta = \conv\{0,\alp(1),\ldots,\alp(n)\}$. Assume that all interior monomials come with a negative coefficient. Then we can write the polynomial as a sum of nonnegative circuit polynomials if and only if it is nonnegative. Furthermore, we get an equivalence between nonnegativity and sums of squares if the whole support is contained in $\Delta^*$. In the following, let $\{\lambda_0^{(i)},\dots,\lambda_{n}^{(i)}\}$ be the (unique) convex combination of $y(i) \in I \subseteq (\Int(\Delta) \cap \N^{n})$ and scale such that $b_0 = \sum_{j=1}^{|I|} \lambda_0^{(j)}$.

\begin{thm}\label{thm:multiple}
 Let $f = \sum_{j=1}^{|I|} \lambda_0^{(j)} + \sum_{j=1}^n b_j \mathbf{x}^{\alpha(j)} - \sum_{y(i)\in I}^{} a_i\mathbf{x}^{y(i)}$ such that $\New(f) = \Delta = \conv\{0,\alpha(1),\dots,\alpha(n)\}$ is a simplex with $\alpha(j) \in (2\N)^n$, all $a_i, b_j > 0$ and $I \subseteq (\Int(\Delta)\cap \N^n)$.
Then
\begin{eqnarray*}
	f\in P_{n,2d} \ \text{ if and only if } \ f \ = \ \sum_{i = 1}^{|I|} E_{y(i)},
\end{eqnarray*}
where all $E_{y(i)} \in P_{\Delta(i)}^{y(i)}$ are nonnegative with support sets $\Delta(i) \subseteq \{0,\alp(1),\ldots,\alp(n),y(i)\}$.

If furthermore $I \subseteq \Delta^*$, then we have
\begin{eqnarray}
	f\in P_{n,2d} & \text{ if and only if } & f\in \Sigma_{n,2d} \label{Equ:PSDequivSOS} \\
	              &	\text{ if and only if } & f \text{ is a sum of binomial squares}. \nonumber
\end{eqnarray}
Particularly, \eqref{Equ:PSDequivSOS} always holds if $\Delta$ is an $H$-simplex.
\end{thm}

Again, we get an immediate corollary.
\begin{cor}
 Let $f$ be as above with $I \subseteq \Delta^*$. Then $f_{sos}^* = f^*$.
\end{cor}

In order to prove Theorem \ref{thm:multiple}, we need the following lemma.

\begin{lemma}
\label{lem:minimizer}
 Let $f = b_0 + \sum_{j=1}^n b_j \mathbf{x}^{\alpha(j)} - \sum_{y(i)\in I}^{} a_i\mathbf{x}^{y(i)}$ be nonnegative with simplex Newton polytope $\New(f) = \Delta = \conv\{0,\alpha(1),\dots,\alpha(n)\}$ for some $\alpha(j) \in (2\N)^n$. Furthermore, let $I \subseteq (\Int(\Delta)\cap \N^n)$ and $a_i, b_j > 0$. Then $f$ has a global minimizer $\mathbf{v}^* \in \R_{>0}^n$.
\end{lemma}

\begin{proof}
Since all $b_j > 0$ and $\alp(j) \in (2\N)^n$, clearly $f$ has a global minimizer on $\R^n$. Assume that all global minimizers are not contained in $\R_{\geq 0}^n$. We make a term by term inspection for a minimizer $\mathbf{v}$ in comparison with $|\mathbf{v}| = (|v_1|,\ldots,|v_n|)$: For every vertex of $\Delta$ we have $b_j \mathbf{v}^{\alpha(j)} = b_j |\mathbf{v}^{\alpha(j)}|$; for every interior point we have $-a_i |\mathbf{v}|^{y(i)} \leq -a_i \mathbf{v}^{y(i)}$ and hence $f(\mathbf{v}) \geq f(|\mathbf{v}|)$. This is a contradiction and therefore at least one global minimizer $\mathbf{v}^*$ is contained in $\R^n_{\geq 0}$.

Assume that for at least one component $v_j^*$ of $\mathbf{v}^*$ it holds that $v_j^* = 0$. We define $g = b_0 + \sum_{j = 1}^n b_j \mathbf{x}^{\alp(j)} - a_i \mathbf{x}^{y(i)}$ for one $y(i) \in I$. By Proposition \ref{Prop:ExtremalPoint}, $g(e^{\mathbf{w}})$ has a unique global minimizer on $\R^n$ and hence $g$ has a unique global minimizer  on $\R_{> 0}^n$. But, by construction of $f$ and $g$, we have $f(\mathbf{x}) < g(\mathbf{x})$ for all $\mathbf{x} \in \R_{> 0}^n$ and $f(\mathbf{x}) = g(\mathbf{x})$ for $\mathbf{x} \in \R_{\geq 0}^n \setminus \R_{> 0}^n$. Thus, $v_j^* \neq 0$ for all $1 \leq j \leq n$.
\end{proof}

\begin{proof}{(Proof of Theorem \ref{thm:multiple})} Let $f = \sum_{j=1}^{|I|} \lambda_0^{(j)} + \sum_{j=1}^n b_j\mathbf{x}^{\alpha(j)} - \sum_{y(i)\in I}^{} a_i\mathbf{x}^{y(i)}$ be nonnegative and, by Lemma \ref{lem:minimizer}, let $\mathbf{v} \in \R^n_{> 0}$ be a global minimizer of $f$. 

First, we investigate the case $\alp(j) = \alp_j e_j$ for some $\alp_j \in 2\N^*$ and $e_j$ denoting the $j$-th standard vector. For any $1 \leq k \leq n$ we have
\begin{eqnarray}
\left(x_k \frac{\partial f}{\partial x_k}\right)(\mathbf{v}) & = & 
b_k \cdot\alpha(k)_k \cdot v_k^{\alpha_k} - \sum_{y(i) \in I}a_i\cdot y(i)_k \cdot \textbf{v}^{y(i)}
\ = \ 0. \label{Equ:DerivativeMultiple}
\end{eqnarray}
Let, again, $\lambda_0^{(i)},\dots,\lambda_{n}^{(i)}$ be the coefficients of the unique convex combination of $y(i) \in I$ and $\lam^{(i)} = (\lam_1^{(i)},\ldots,\lam_n^{(i)}) \in \R_{> 0}^n$. For $y(i) \in I$ we define 
\begin{eqnarray}
	b_{y(i),k} & = & \frac{a_i\cdot \lambda_k^{(i)}\cdot \mathbf{v}^{y(i)}}{\mathbf{v}^{\alpha(k)}}.
\label{Equ:Defby(i)}
\end{eqnarray}
Since for all $i$ and all $k$ it holds that $\sum_{j = 1}^n \lam^{(i)}_k \alp(j)_k = y(i)_k$ and that all $\alp(j)_k = 0$ unless $j = k$, we obtain with \eqref{Equ:DerivativeMultiple} that
\begin{eqnarray*}
	b_k & = & \sum_{y(i)\in I} b_{y(i),k}. 
\end{eqnarray*}
 By Proposition \ref{Prop:GlobalMinimizer} and Theorem \ref{Thm:Positiv}, we conclude that
\begin{eqnarray*}
	E_{y(i)}(\mathbf{x}) & = & \lambda_0^{(i)} + \sum_{k=1}^n b_{y(i),k}x_k^{\alpha_k} - a_i \mathbf{x}^{y(i)} 
\end{eqnarray*}
is a nonnegative circuit polynomial and has its minimum value at $\mathbf{v}$. We obtain
\begin{eqnarray}
 f(\mathbf{x}) &=& \sum_{j=1}^{|I|} \lambda_0^{(j)} + \sum_{k=1}^n b_k x_k^{\alpha_k} - \sum_{y(i)\in I}a_i \mathbf{x}^{y(i)} \nonumber\\
 &=& \sum_{j=1}^{|I|} \lambda_0^{(j)} + \sum_{k=1}^n \left(\sum_{y(i)\in I} b_{y(i), k}\right)x_k^{\alpha_k} - \sum_{y(i) \in I}a_i \mathbf{x}^{y(i)} \label{Equ:GeneralSimplexSONCDecomp}\\
 &=&  \sum_{y(i) \in I} E_{y(i)}(\mathbf{x}). \nonumber
\end{eqnarray}

Now, we consider the case of arbitrary $\alp(j) \in (2\N)^n$. Let $\mathbf{v} \in \R_{> 0}^n$ be a global minimizer of $f$.
By Corollary \ref{Cor:StandardFormGeneralSimplex} (and Proposition \ref{Prop:StandardForm}) there exists a unique polynomial $g$ satisfying 
\begin{eqnarray}
	f(e^{\mathbf{w}}) & = & g(e^{T^t \mathbf{w}}) \ \text{ for all } \ \mathbf{w} \in \R^n \label{Equ:TransformationGeneralSimplex}
\end{eqnarray}
such that $T \in GL_n(\R)$ and $g$ has a support matrix 
\begin{eqnarray*}
M^{A'} & = & \left(
\begin{array}{cccccccc}
1 & 1 & \cdots & \cdots & 1 & 1 & \cdots & 1\\
0 & \mu & 0 & \cdots & 0 & \mu \lam_1^{(1)} & \cdots & \mu \lam_1^{(|I|)} \\
\vdots & 0 & \ddots & & \vdots & \vdots & \cdots & \vdots \\
\vdots & \vdots & & \ddots & 0 & \vdots & \cdots & \vdots \\
0 & 0 & \cdots & 0 & \mu & \mu \lam_n^{(1)} & \cdots & \mu \lam_n^{(|I|)} \\
\end{array}\right)
\ \in \ \Mat(\Z,(n+1) \times (n+ |I|)),
\end{eqnarray*}
where $\mu$ is the least common multiple of the denominators of all $\lam_j^{(i)}$ and 2 (since vertices of $\New(g)$ shall be in $(2\N)^n$).

Since $\mathbf{v} \in \R_{> 0}^n$, we can define $\Log|\mathbf{v}'| = T^{t} \Log|\mathbf{v}|$. By 
\eqref{Equ:GeneralSimplexSONCDecomp} and \eqref{Equ:TransformationGeneralSimplex} it follows that $\mathbf{v}'$ is a global minimizer for $g$ and thus we have
\begin{eqnarray*}
	f(\mathbf{v}) \ = \ f(e^{\Log|\mathbf{v}|}) \ = \ g(e^{T^{t}\Log|\mathbf{v}|}) \ = \ \sum_{i = 1}^{|I|} E_{\mu \lam^{(i)}}(e^{\Log|\mathbf{v}'|}),
\end{eqnarray*}
for some nonnegative circuit polynomials $E_{\mu \lam^{(i)}}$ with global minimizer $\mathbf{v}' \in \R_{> 0}^n$.

Since $\supp(E_{\mu \lam^{(i)}}) \subseteq \supp(g)$ and $\New(E_{\mu \lam^{(i)}}) = \New(g)$, we have, by Proposition \ref{Prop:GlobalMinimizer},
\begin{eqnarray*}
	E_{\mu \lam^{(i)}}(e^{\Log|\mathbf{v}'|}) \ = \ E_{y(i)}(e^{\Log|\mathbf{v}|})
\end{eqnarray*}
such that each $E_{y(i)}(e^{\Log|\mathbf{v}|})$ is a nonnegative circuit polynomial with global minimizer $\mathbf{v}$ and support set  $\{0,\alp(1),\ldots,\alp(n),y(i)\}$ satisfying $f = \sum_{i = 1}^{|I|} E_{y(i)}$.\\

If, additionally, every $y(i) \in \Delta^*$ (for example if $\Delta$ is an $H$-simplex), then we know by Theorem \ref{Thm:MainSOS} that all $E_{y(i)}(\mathbf{x})$ are sums of (binomial) squares and, hence, $f$ is a sum of (binomial) squares.
\end{proof}

Note that Theorem \ref{thm:multiple} generalizes \cite[Theorem 2.7]{Fidalgo:Kovacec}, where an analog statement is shown for the special case of \emph{diagonal minus tail} forms $f$, which are given by $\alpha(j) = 2d$ for $1 \leq j\leq n$.

We remark that the correct decomposition of the $b_j$ in Theorem \ref{thm:multiple} for the case of a general simplex Newton polytope is also given by \eqref{Equ:Defby(i)}, since due to 
\begin{eqnarray*}
e^{\langle \Log|\mathbf{v}|,y(i) - \alp(j) \rangle} \ = \ e^{\langle (T^{t})^{-1} \Log|\mathbf{v}'|, T^t(\mu (\lam^{(i)} - e_j))\rangle} \ = \ e^{\langle \Log|\mathbf{v}'|, \mu (\lam^{(i)} - e_j)\rangle}
\end{eqnarray*}
these scalars remain invariant under the transformation $T$ from and to the standard form.

\begin{example}
The polynomial $f = 1 + \frac{1}{2}x^6 + \frac{1}{32}y^4 - \frac{1}{2}xy - \frac{1}{2}x^2y$ is nonnegative and has a zero at $\mathbf{v} = (1,2)$. By using the constructions in Theorem \ref{thm:multiple}, we can decompose $f$ as sum of two polynomials in $P_{n,2d}^y$ with $y \in \{(1,1),(2,1)\}$ and vanishing at $\mathbf{v}$. More precisely,
$$f \ = \ \left(\frac{7}{12} +\frac{1}{6} x^6 + \frac{1}{64}y^4 - \frac{1}{2}xy\right) + \left(\frac{5}{12} +\frac{1}{3} x^6 + \frac{1}{64}y^4 - \frac{1}{2}x^2y\right).$$
Since $\Delta$ is an $H$-simplex, we have $f \in \Sigma_{2,6}$. Using the algorithm in \cite{Reznick:AGI} and a suitable variable transformation (see proof of Theorem \ref{thm:sos}), we get the following representation for $f$ as a sum of binomial squares:
$$f \ = \ \frac{1}{2}(x - x^3)^2 + \frac{1}{2}\left(\frac{1}{2}y - x\right)^2 +\frac{1}{2}\left(\frac{1}{2}y - x^2\right)^2 + \frac{1}{2}\left(1 - x^2\right)^2 + \frac{1}{2}\left(1 - \frac{1}{4}y^2\right)^2.$$
\end{example}

\subsection{A Sufficient Condition for \textit{H}-simplices}
By Theorem \ref{thm:sos}, all nonnegative polynomials in $P_{\Delta}^y$ supported on an $H$-simplex are sums of squares. Here, we provide a sufficient condition for a lattice simplex $\Delta$ to be an $H$-simplex, meaning, that all lattice points in $\Delta$ except the vertices are midpoints of two even distinct lattice points in $\Delta$. In the following, we call a full dimensional lattice polytope $P \subset \R^n$ $k$-\emph{normal}, if every lattice point in $kP$ is a sum of exactly $k$ lattice points in $P$, i.e.,
$$k \in \N, m \in kP \cap \Z^n \ \Rightarrow \ m = m_1 +\ldots + m_k, \quad m_1,\dots, m_k \in P \cap \Z^n.$$
For an introduction to toric ideals, see for example \cite{Sturmfels:toric}.

\begin{thm}
\label{Thm:toric}
Let $\hat\Delta = \{\alpha(0), \alpha(1), \dots, \alpha(n)\}\subset (2\mathbb N)^n$ and $\Delta = \conv(\hat \Delta)$ be a lattice simplex. Furthermore, let $B = \frac{1}{2}\Delta \cap \mathbb N^n$ and $I_B$ be the corresponding toric ideal of $B$. If
\begin{enumerate}
 \item $I_B$ is generated in degree two, i.e., $I_B = \langle I_{B,2} \rangle$ and
 \item the simplex $\frac{1}{2}\Delta$ is 2-normal,
\end{enumerate}
then $\Delta$ is an $H$-simplex.
\end{thm}

\begin{proof}
Let $L = (\Delta \cap \mathbb N^n)\setminus \hat \Delta$. Note that for $u \in L\setminus (2\mathbb N)^n$ the statement follows from normality of $\frac{1}{2}\Delta$, since we have $u = s + t$ with $s,t\in B$. Therefore, $u = \frac{2s + 2t}{2}$. Now, let 
$$ \left\{\frac{1}{2}\alpha(0),\dots,\frac{1}{2}\alpha(n)\right\} = \{\alpha(0)',\dots,\alpha(n)'\}$$ 
be the vertices of $\frac{1}{2}\hat \Delta$ and consider $u \in B\setminus \frac{1}{2}\hat \Delta$. By clearing denominators in the unique convex combination of $u$ we get a relation
$$N\cdot u \ = \ \lambda_0\alpha(0)'+\dots + \lambda_n\alpha(n)',\quad N = \sum_{i=0}^n \lambda_i,\quad \lambda_i \geq 0.$$
For the corresponding toric ideal $I_B$, this implies that $x_u^N - \prod_{i=0}^nx_{\alpha(i)'}^{\lambda_i} \in I_B$. Since $I_B$ is generated in degree two, we have the following representation:
$$x_u^N - \prod_{i=0}^nx_{\alpha(i)'}^{\lambda_i} = \sum_{\substack{m,n \in \mathbb N^B\\|m|=|n|=2}}^{} f_{m,n}(x^m - x^n)$$
for some polynomials $f_{m,n}$. Matching monomials, it follows that there exists $m$  such that $x^m = x_u^2$ (note that $f_{m,n}$ contains $x_u^{N-2}$). Since $|m| = 2$, we have $x_u^2 - x_vx_{v'} \in I_B$ with $v,v' \in B$, yielding the relation $2u = \frac{2v + 2v'}{2}$, i.e., $2u$ is a convex combination of two even lattice points $2v$ and $2v'$. 
\end{proof}

\begin{cor}\label{cor:R2}
 Let $\Delta\subset\mathbb R^2$ be a lattice simplex as in Theorem \ref{Thm:toric} such that $\frac{1}{2}\Delta$ has at least four boundary lattice points. Then $\Delta$ is an $H$-simplex.
\end{cor}

\begin{proof}
 Since every $2$-polytope is normal, we only need to prove that the corresponding toric ideal is generated in degree two. But this is \cite[Theorem 2.10]{Koelman:toric}.
\end{proof}

Hence, in $\mathbb R^2$, almost every simplex $\Delta$ corresponding to $P_{\Delta}^y$ is an $H$-simplex, which is a fact that was announced in \cite{Reznick:AGI} without proof. This implies that the sections $P_{2,2d}^y$ and $\Sigma_{2,2d}^y$ almost always coincide.

\begin{example} 
We demonstrate Theorem \ref{Thm:toric} by two interesting examples.
\begin{enumerate}
 \item   The Newton polytope of the Motzkin polynomial $$m = 1 + x^4y^2 + x^2y^4 - 3x^2y^2 \in P_{2,6}\setminus\Sigma_{2,6}$$ is an $M$-simplex $\Delta = \conv\{(0,0),(4,2),(2,4)\}$ such that $\frac{1}{2}\Delta$ has exactly three boundary lattice points. One can check that the corresponding toric ideal $I_B$ is generated by cubics.
\item Note that the conditions in Theorem \ref{Thm:toric} are not equivalent. The lattice simplex $\Delta = \conv\{(0,0),(2,4),(10,6)\}$ is easily checked to be an $H$-simplex, but $\partial\frac{1}{2}\Delta$ contains exactly three lattice points.
\end{enumerate}
\end{example}

In higher dimensions things get more involved both in checking the conditions in Theorem \ref{Thm:toric} and in determining the maximal $\hat\Delta$-mediated set $\Delta^*$. Note that $\Delta^*$ can lie strictly between  $A(\hat\Delta)$ and $\Delta \cap \Z^n$, which correspond to $M$-simplices and $H$-simplices. In \cite{Reznick:AGI} an algorithm for the computation of $\Delta^*$ is given. One expects the existence of better algorithms, but, to our best knowledge, no more efficient algorithm is known. On the other hand, checking normality of polytopes and quadratic generation of toric ideals is an active area of research. It is an open problem to decide, whether every smooth lattice polytope is normal and the corresponding toric ideal is generated by quadrics, see \cite{Gubeladze:normality, Sturmfels:toric}. However, for an arbitrary lattice polytope $P$ the multiples $kP$ are normal for $k\geq \dim P - 1$ and their toric ideals are generated by quadrics for $k \geq \dim P$ \cite{Bruns:et:al}. In light of these results, we can conclude another interesting corollary from Theorem \ref{Thm:toric}.

\begin{cor}\label{cor:bound}
 Let $\Delta \subset\R^n$ be a lattice simplex as in Theorem \ref{Thm:toric} such that $\frac{1}{2}\Delta = M\Delta'$ for a lattice simplex $\Delta'\subset\R^n$ and $M\geq n$. Then $\Delta$ is an $H$-simplex.
\end{cor}

\begin{proof}
The result follows from the previously quoted results together with Theorem \ref{Thm:toric}.
\end{proof}

Note that Corollaries \ref{cor:R2} and \ref{cor:bound} yield large sections at which nonnegative polynomials and sums of squares coincide.

\section{Convex Polynomials and Forms Supported on Circuits}
\label{Sec:convex}
In this section, we investigate convex polynomials and forms (i.e., homogeneous polynomials) supported on a circuit. Recently, there is much interest in understanding the convex cone of convex polynomials/forms. Since deciding convexity of polynomials is NP-hard in general \cite{Ahmadi:et:al:convex}, but very important in different areas in mathematics, such as convex optimization, the investigation of properties of the cones of convex polynomials and forms is a genuine problem.

\begin{defn}
 Let $f \in \mathbb R[\mathbf{x}]$. Then $f$ is \emph{convex} if the Hessian $H_f$ of $f$ is positive semidefinite for all $\mathbf{x} \in \mathbb R^n$, or, equivalently, $\mathbf{v}^t H_f(\mathbf{x}) \mathbf{v} \geq 0$ for all $\mathbf{x},\mathbf{v} \in \mathbb R^n$.
\end{defn}

Unlike the property of nonnegativity and sums of squares, convexity of polynomials is not preserved under homogenization. Therefore, we need to distinguish between convex polynomials and convex forms. The relationship between convexity on the one side and nonnegativity and sums of squares on the other side arises when considering homogeneous polynomials, since every convex form is nonnegative. However, the relation between convex forms and sums of squares is not well understood except for the fact that their corresponding cones are not contained in each other. The problem to find a convex form that is not a sum of squares is still open. For an overview and proofs of the previous facts see \cite{Blekherman:Parrilo:Thomas, Reznick:Blenders}. Here we investigate convexity of polynomials and forms in the class $P_{\Delta}^y$. We start with the univariate (nonhomogeneous) case.

\begin{prop}
\label{prop:convexunivariat}
 Let $f = 1 + ax^y + bx^{2d} \in P_{\Delta}^y$ and $b > 0$. Then $f$ is convex exactly in the following cases.
\begin{enumerate}
 \item $y = 1$,
 \item $a \geq 0$ and $y=2l$ for $y > 1$ and $l \in \N$.
\end{enumerate}
\end{prop}

\begin{proof}
 Let $f = 1 + ax^y + bx^{2d}$. Note that the degree is necessarily even and $b>0$. $f$ is convex if and only if $D^2(f) \geq 0$ where
$D^2(f) = ay(y-1)x^{y-2} + 2db(2d-1)x^{2d-2}$. For $y=1$ the polynomial $D^2(f)$ is a square and hence $f$ is convex. Now, consider the case $y > 1$. First, suppose that $a < 0$. Then $D^2(f)$ is always indefinite, since the monomial $x^{y-2}$ in $D^2(f)$ corresponds to a vertex of the corresponding Newton polytope of $D^2(f)$ and has a negative coefficient. Otherwise, if $a \geq 0$ and $y=2l$ for $l \in \N$, then $D^2(f)\geq 0$ and $f$ is convex. If $y=2l+1$, then $x^{y-2}$ has an odd power and hence $D^2(f)$ is indefinite, implying that $f$ is not convex.
\end{proof}

The homogeneous version is much more difficult than the affine version. We just prove the following claims instead of giving a full characterization.

\begin{prop}
 Let $f = z^{2d} + ax^yz^{2d-y} + bx^{2d} \in P_{\Delta}^y$ be a form and $b > 0$. Then the following hold.
\begin{enumerate}
 \item For $y = 2l - 1$, $l\in \mathbb N$, or $a \leq 0$, the form $f$ is not convex.
\item For $y = 2l$ and $0 \leq a \leq \frac{(y-1)(2d-y-1)}{y(2d-y)}$ the form $f$ is convex.
\end{enumerate}
\end{prop}

\begin{proof}
We have 
$$\frac{\partial^2f}{\partial z^2} \ = \ 2d(2d - 1)z^{2d-2} + (2d - y)(2d - y -1)ax^yz^{2d-y-2}.$$ 
Evaluating this partial derivative at $z = 1$, in order to be nonnegative, it is obvious that $y$ must be even and $a \geq 0$, proving the first claim. For the second claim, we investigate the principal minors of $H_f$. We have that $\frac{\partial^2f}{\partial x^2}\geq 0$ if and only if $D^2(f)\geq 0$ where $D^2(f)$ is the dehomogenized polynomial $\frac{\partial^2f}{\partial x^2}(x,1)$. This yields $y = 1$ or $a \geq 0$ and $y = 2l$. From $\frac{\partial^2f}{\partial z^2}$ we get again that $y$ must be even and $a \geq 0$. Finally, one can check that all exponents of the dehomogenized determinant $\det H_f(x,1)$ are even and have positive coefficients for $0 \leq a \leq \frac{(y-1)(2d-y-1)}{y(2d-y)}$. Hence, for $y = 2l$ and $0 \leq a \leq \frac{(y-1)(2d-y-1)}{y(2d-y)}$ the form $f$ is convex.
\end{proof}

Note that for $y = 1$ the form $f = z^{2d} + ax^yz^{2d-y} + bx^{2d} \in P_{\Delta}^y$ is never convex, whereas, by Proposition \ref{prop:convexunivariat}, the dehomogenized polynomial is always convex.
As a sharp contrast, we prove the surprising result that for $n\geq 2$ there are no convex polynomials in the class $P_{\Delta}^y$, implying that there are no convex forms in $P_{\Delta}^y$ for $n \geq 3$.

\begin{thm}
\label{thm:konvex}
 Let $n\geq 2$ and $f \in P_{\Delta}^y$. Then $f$ is not convex.
\end{thm}

\begin{proof}
 Let $$f = 1 + \sum_{j=1}^n A_jx_1^{\alpha(j)_1}\cdot\ldots\cdot x_n^{\alpha(j)_n} + Bx_1^{y_1}\cdot\ldots\cdot x_n^{y_n}$$

with $A_j > 0$ for $1\leq j \leq n$ and $B \in \mathbb R^*$. We will prove that the principal minor $[1,2]\times [1,2]$ (deleting all rows and columns except the first and second one) of the Hessian of $f$ is indefinite, implying that the Hessian of $f$ is not positive semidefinite and, hence, the polynomial $f$ is not convex. We have
{\tiny
\begin{eqnarray*}
 \frac{\partial^2 f}{\partial x_1^2}\frac{\partial^2 f}{\partial x_2^2}-\left(\frac{\partial^2 f}{\partial x_1x_2}\right)^2 &=&\sum_{j=1}^n\sum_{i=1}^n\left(\alpha(j)_1(\alpha(j)_1-1)A_jx_1^{\alpha(j)_1-2}x_2^{\alpha(j)_2}\cdot\ldots\cdot x_n^{\alpha(j)_n} + y_1(y_1-1)Bx_1^{y_1-2}x_2^{y_2}\cdot\ldots\cdot x_n^{y_n}\right)\\
&\cdot& \left(\alpha(i)_2(\alpha(i)_2-1)A_ix_1^{\alpha(i)_1}x_2^{\alpha(i)_2-2}\cdot\ldots\cdot x_n^{\alpha(n)_i}+By_2(y_2-1)x_1^{y_1}x_2^{y_2-2}x_3^{y_3}\cdot\ldots\cdot x_n^{y_n}\right)\\ 
&-&\left(\sum_{k=1}^n\alpha(k)_1\alpha(k)_2A_kx_1^{\alpha(k)_1-1}x_2^{\alpha(k)_2-1}x_3^{\alpha(k)_3}\cdot\ldots\cdot x_n^{\alpha(k)_n} + By_1y_2x_1^{y_1-1}x_2^{y_2-1}x_3^{y_3}\cdot\ldots\cdot x_n^{y_n}\right)^2.\\
\end{eqnarray*}
}
We claim that there is a point $\mathbf{x} \in \mathbb R^n$ at which this minor is negative. For this, note that all exponents in $\left(\frac{\partial^2 f}{\partial x_1x_2}\right)^2$ are captured by those in $\frac{\partial^2 f}{\partial x_1^2}\frac{\partial^2 f}{\partial x_2^2}$. Hence, we can restrict to the latter ones. The $\binom{n+2}{2}$ different exponents are of the following type:
\begin{enumerate}
 \item $(2\alpha(j)_1 - 2, 2\alpha(j)_2 - 2,2\alp(j)_3,\ldots,2\alp(j)_n)$ for $1\leq j\leq n$,
\item $(\alpha(i)_1 + \alpha(j)_1 - 2, \alpha(i)_2 + \alpha(j)_2 - 2,\alp(i)_3 + \alp(j)_3,\ldots,\alp(i)_n + \alp(j)_n)$ for $1\leq i < j\leq n$,
\item $(\alpha(j)_1 + y_1 - 2, \alpha(j)_2 + y_2 - 2, \alp(j)_3 + y_3,\ldots,\alp(j)_n + y_n)$ for $1\leq j\leq n$,
\item $(2y_1 - 2, 2y_2 - 2,2y_3,\ldots,2y_n)$.
\end{enumerate}
We claim that the point $(2y_1 - 2, 2y_2 - 2,2y_3,\ldots,2y_n)$ is always a vertex in the convex hull of the points (1)-(4), i.e., in the Newton polytope of the investigated minor. The points in (2) are obviously convex combinations from appropriate points in (1) and the points in (3) are convex combinations from points in (1) and (4). Hence, it remains to show that (4) is not a convex combination of the points in (1). Therefore, denote the points in (1) by $P_j$ and the point in (4) by $Q$. Let
\begin{eqnarray*}
	Q & = & \sum_{j=1}^n \mu_j P_j \ \text{ with } \ \sum_{j=1}^n \mu_j = 1 \ \text{ and } \ \mu_j\geq 0 \ \text{ for all } \ 1\leq j\leq n.
\end{eqnarray*}

But since $\sum_{j = 1}^n \mu_j (-2) = -2$, this equation is equivalent to
\begin{eqnarray*}
          y & = & \sum_{j=1}^n \mu_j\alpha(j) \ \text{ with } \ \sum_{j=1}^n \mu_j = 1  \ \text{ and } \ \mu_j \geq 0  \ \text{ for all } \ 1\leq j\leq n.
\end{eqnarray*}
 
But this means that $y$ lies on the boundary of $\Delta$, the Newton polytope of $f$. This is a contradiction, since $f \in P_{\Delta}^y$, i.e., $y\in \Int(\Delta)$. Hence, (4) is a vertex of the Newton polytope of the investigated minor. Extracting the coefficient of its corresponding monomial in the minor, we get that this coefficient equals $-B^2y_1y_2(y_1 + y_2 - 1) < 0$. Therefore, the Newton polytope of the minor of the Hessian of $f$ has a vertex coming with a negative coefficient and, hence, it is indefinite, proving the claim.
\end{proof}

Note that this already implies that there is also no convex form in $P_{\Delta}^y$ whenever $n\geq 3$, since non-convexity is preserved under homogenization. Since it is mostly unclear which structures prevent polynomials from being convex, Theorem \ref{thm:konvex} is an indication that sparsity is among these structures.

\section{Sums of Nonnegative Circuits}
\label{Sec:SOPC}
Motivated by results in previous sections, we recall Definition \ref{Def:SONC} from the introduction, where we introduced sums of nonnegative circuit polynomials (SONC's), a new family of nonnegativity certificates.

\begin{defn}
 We define the set of \emph{sums of nonnegative circuit polynomials} (SONC) as
$$ C_{n,2d} \ = \ \left\{f \in \R[\mathbf{x}]_{2d} \ :\  f = \sum_{i=1}^k \lambda_i g_i, \lambda_i \geq 0, g_i \in P_{\Delta_i}^y\cap P_{n,2d}\right\}$$
for some even lattice simplices $\Delta_i \subset \R^n$.
\end{defn}

Remember that membership in $P_{n,2d}^y$ can easily be checked and is completely characterized by the circuit numbers $\Theta_{g_i}$ (Theorem \ref{Thm:Positiv}). Obviously, for $\alpha,\beta \in \R_{>0}$ and $f,g \in C_{n,2d}$, it holds that $\alpha f + \beta g \in C_{n,2d}$, hence, $C_{n,2d}$ is a convex cone. Then we have the following relations.

\begin{prop}
\label{prop:cones}
 The following relationships hold between the corresponding cones.
\begin{enumerate}
 \item $C_{n,2d} \subset P_{n,2d}$ for all $d,n \in \N$,
\item $C_{n,2d} \subset \Sigma_{n,2d}$ if and only if $(n,2d)\in\{(1,2d),(n,2),(2,4)\}$,
\item  $\Sigma_{1,2} \subset C_{1,2}$ and $\Sigma_{n,2d} \not\subset C_{n,2d}$ for all $(n,2d)$ with $2d \geq 6$.
\item $C_{n,2d} \cap K_{n,2d} = \{0\}$ for $n \geq 2$, where $K_{n,2d}$ denotes the cone of convex polynomials.
\end{enumerate}
\label{Prop:ConeContainment}
\end{prop}

\begin{proof}
 Since all $\lambda_i g_i \in P_{n,2d}$, the first inclusion is obvious. For the second part note that one direction follows from the first inclusion and Hilbert's Theorem \cite{Hilbert:Seminal} stating that $(n,2d)\in\{(1,2d),(n,2),(2,4)\}$ if and only if $P_{n,2d} = \Sigma_{n,2d}$. Conversely, if $(n,2d)\notin\{(1,2d),(n,2),(2,4)\}$ then one can use homogenizations of the Motzkin polynomial and the dehomogenized agiform $N = 1 + x^2y^2 + y^2z^2 + x^2z^2 - 4xyz \in P_{3,4}\setminus\Sigma_{3,4}$ to obtain polynomials in $C_{n,2d}\setminus\Sigma_{n,2d}$. 
 
 Considering $(3)$ note that if $(n,2d) = (1,2)$ then $\Sigma_{1,2} = P_{1,2} = C_{1,2}$. In other cases we make use of the following observations. By Corollary \ref{cor:zerobound}, a polynomial $f \in C_{n,2d}$ has at most $2^n$ zeros. Additionally, by \cite[Proposition 4.1]{Reznick:realzeros} there exist polynomials in $\Sigma_{n,2d}$ with $d^n$ zeros. The only cases, for which the claim does not follow by this argument is the case $(n,2d) = (n,4)$. $(4)$ follows from Theorem \ref{thm:konvex}.
\end{proof}

Hence, the convex cone $C_{n,2d}$ serves as a nonnegativity certificate, which, by Proposition \ref{prop:cones}, is independent from sums of squares certificates.

\begin{example}
 Let $f = 3 + 4y^4 + 6x^8 + x^4y^4 - 3xy + 5x^3y + 2x^4y^2$. The Newton polytope $\New(f) = \conv\{(0,0)^T, (0,4)^T, (4,4)^T, (8,0)^T)\}$ is not a simplex and $f \in C_{2,8}$. An explicit representation is given by
$$f \ = \ (1 + 2x^8 + 2y^4 - 3xy) + (1 + 3x^8 + 2y^4 + 5x^3y) + (1 + x^8 + x^4y^4 + 2x^4y^2).$$
\end{example}

We give two further remarks about the Proposition \ref{Prop:ConeContainment}:

\begin{enumerate}
 \item As stated in the proof $(n,4)$ is the case, which is not covered in Part (3). We believe that $\Sigma_{n,4} \not\subset C_{n,4}$ for all $n$ but we do not have an example.
 \item Let $\wh C_{n,2d}$ be the subset of $C_{n,2d}$ containing all polynomials with a full dimensional Newton polytope. It is not obvious for which cases next to $(n,2d) \in\{(1,2d),(n,2),(2,4)\}$ it holds that $\wh C_{n,2d} \subseteq \Sigma_{n,2d}$. However, $\wh C_{n,2d} \not\subseteq \Sigma_{n,2d}$ if we require $d \geq n+1$ as we show in the following example.
\end{enumerate}

\begin{example}
 Let $f = 1 + \sum_{j = 1}^n \mathbf{x}^{\alp(j)} - c \cdot \mathbf{x}^{(2,\ldots,2)}$ with $\alp(j) = (2,\ldots,2) + 2 \cdot e_j$ where $e_j$ denotes the $j$-th unit vector and $n+1 \leq c < 0$. By Theorem \ref{Thm:Positiv} we conclude that $f$ is a nonnegative circuit polynomial in $n$ variables of degree $2n+2$. Hence, $f \in C_{n,2d}$ for all $n$ and $d \geq n+1$. Moreover, $\New(f)$ is an $n$-dimensional polytope by construction.  But $f \notin \Sig_{n,2d}$. Namely, it is easy to see that the simplex $1/2 \cdot \New(f) = \conv\{0,1/2 \cdot \alp(1),\ldots,1/2 \cdot \alp(n)\}$ only contains the lattice point $(1,\ldots,1)$ in the interior. Therefore, $\New(f)$ has exactly one even lattice point in the interior, the point $(2,\ldots,2)$. It follows from a statement by Reznick \cite[Theorem 2.5]{Reznick:AGI} that $\New(f)$ is an $M$-simplex. Hence, $f \notin \Sig_{n,2d}$ by Theorem \ref{Thm:MainSOS}.
\end{example}

Of course, a priori it is completely unclear for which type of nonnegative polynomials a SONC decomposition exists and how big the gap between $C_{n,2d}$ and $P_{n,2d}$ is. Furthermore, it is not obvious how to compute such a decomposition, if it exists. We discuss this question in a follow up article \cite{Iliman:deWolff:GP}. In this article we show in particular that for simplex Newton polytopes (with arbitrary support) such a decomposition exists if and only if a particular geometric optimization problem is feasible, which can be checked very efficiently. This generalizes similar results by Ghasemi and Marshall \cite{Ghasemi:Marshall:GPGlobal,Ghasemi:Marshall:GPSemialgebraic}. Here we deduce as a fruitful first step the following corollary from Theorem \ref{thm:multiple}.

\begin{cor}
Let $f = b_0 + \sum_{j = 1}^n b_j \mathbf{x}^{\alp(j)} + \sum_{i = 1}^k a_i \mathbf{x}^{y(i)}$ be nonnegative with $b_j \in \R_{> 0}$ and $a_i \in \R^*$ such that $\New(f) = \Delta = \conv\{0,\alp(1),\ldots,\alp(n)\}$ is a simplex and all $y(i) \in (\Int(\Delta) \cap \N^n)$. If there exists a vector $\mathbf{v} \in (\R^*)^n$ such that $a_i \mathbf{v}^{y(i)} < 0$ for all $1 \leq i \leq k$, then $f$ is SONC. 
\label{Cor:SONC}
\end{cor}

\begin{proof}
Every monomial square is a strictly positive term as well as a $0$-simplex circuit polynomial. Thus, we can ignore these terms. If a particular vector $\mathbf{v} \in (\R^*)^n$ with the desired properties exists, then Theorem \ref{thm:multiple} immediately yields a SONC decomposition after a variable transformation $x_j \mapsto -x_j$ for all $j$ with $v_j < 0$.
\end{proof}

\section{Extension to Arbitrary Polytopes and Counterexamples}
\label{Sec:conjecture}
In Section \ref{Sec:PSDSOS} we proved for $f \in P_{\Delta}^y$ that $f\in \Sigma_{n,2d}^y$ if and only if $y \in \Delta^*$ or $f$ is a sum of monomial squares. One might wonder whether this equivalence also holds for arbitrary polytopes. More precisely, let $Q \subset \R^n$ be an arbitrary lattice polytope and denote by $AP_{Q}^y$ the set of all polynomials of the form $\sum_{\alp \in \V(Q)} b_\alp \mathbf{x}^{\alp} + c \mathbf{x}^y$ that are supported on the vertices $\V(Q)$ of $Q$ and an additional interior lattice point $y \in \Int(Q)$. As a generalization of our previous notation, we call $f \in AP_Q^y$ an \emph{agiform} if $\sum_{\alp \in \V(Q)} b_\alp \alp = y$ and $\sum_{\alp \in \V(Q)} b_\alp = 1$ as well as $b_\alp > 0$ and $c = -1$.

In \cite[Section 10]{Reznick:AGI}, it is asked, whether the lattice point criterion $y \in Q^*$ is again an equivalent condition for a polynomial in $AP_Q^y$ to be a sum of squares. And, if not, how sums of squares can be characterized in this case. Here, we provide a solution to this question (Theorem \ref{Thm:ReznickProblem}). Let $P_{Q}^y$ respectively $\Sigma_{Q}^y$ denote the set of nonnegative respectively sums of squares polynomials in $AP_Q^y$. As for a simplex $\Delta$, for an arbitrary lattice polytope $Q$, we use the same definition of an $M$-polytope respectively an $H$-polytope.\\

The implication $f\in \Sigma_{Q}^y \Rightarrow y \in Q^*$ does always hold. For agiforms, this is proven already in \cite{Reznick:AGI}. The proof in the case of arbitrary coefficients follows exactly the same line as the proof of Theorem \ref{thm:sos}. 

\begin{prop}
\label{prop:counterexample}
There exists $f\in P_{Q}^y\setminus \Sigma_{Q}^y$ and $y \in Q^*$.
\end{prop}

\begin{proof}
 We provide an explicit example. Let 
$$Q \ = \ \conv\{v_0,v_1,v_2,v_3\} \ = \ \conv\{(0,0), (4,0), (4,2), (2,4)\} \ \text{ with } \ y \ = \ (2,2).$$

It is easy to check that $Q$ is an $H$-polytope (indeed, it can actually be proven that Theorem \ref{Thm:toric} is true for arbitrary polytopes not just for simplices).
Since $Q$ is not a simplex, there are infinitely many convex combinations of $y$:
$$y \ = \ \lambda_0v_1 + \lambda_1v_1 + \lambda_2v_2 + \lambda_3v_3 \ \text{ such that } \ \sum_{i=0}^3 \lambda_i = 1 \ \text{ and } \ \lambda_i \geq 0.$$
The set of convex combinations of $y$ is given by $$\left\{(\lambda_0,\lambda_1,\lambda_2,\lambda_3) = \left(\frac{1}{2} - \frac{1}{2}\lambda_3, -\frac{1}{2} + \frac{3}{2}\lambda_3,  1 - 2\lambda_3, \lambda_3\right) \ : \ \frac{1}{3}\leq \lambda_3 \leq \frac{1}{2}\right\}.$$
The corresponding agiform $f(Q,\lambda,y)$ is then given by
$$f(Q,\lambda,y) \ = \ \left(\frac{1}{2} - \frac{1}{2}\lambda_3\right) + \left(-\frac{1}{2} + \frac{3}{2}\lambda_3\right)x^4 + ( 1 - 2\lambda_3)x^4y^2 + \lambda_3x^2y^4 - x^2y^2.$$
For $\lambda_3 = \frac{2}{5}$, the nonnegative polynomial 
$$f \ = \ \frac{3}{10} +  \frac{1}{10}x^4 +  \frac{1}{5}x^4y^2 +  \frac{2}{5}x^2y^4 - x^2y^2$$ 
can easily be checked to be not a sum of squares although $y \in Q^*$ via the corresponding Gram matrix.
\end{proof}

Actually, one can prove that the polynomial $f(Q,\lambda,y)$ in the above proof is a sum of squares if and only if $\lambda_3 = \frac{1}{2}$. In \cite{Reznick:AGI}, the author suspects that the condition $y \in Q^*$ is not sufficient by looking at similar examples. However, in all of these examples, the constructed polynomials that are nonnegative but not a sum of squares are not supported on the vertices of $Q$ and an additional interior lattice point $y \in \Int(Q)$. We conclude that in the non-simplex case the problem of deciding the sums of squares property depends on the coefficients of the polynomials, a sharp contrast to the simplex case. However, motivated by a question in \cite{Reznick:AGI} for agiforms, we are interested in the following sets: Let $C(y)$ denote the set of convex combinations of the interior lattice point $y \in \Int(Q)$, i.e.,
$$C(y) \ = \ \left\{\lambda = (\lambda_0,\dots,\lambda_s) \ : \ y = \sum_{i=0}^s \lambda_iv_i,\,\,\sum_{i=0}^s \lambda_i = 1,\,\, \lambda_i \geq 0\right\}$$
where $v_i$ are the $s$ vertices of $Q$. Note that $C(y)$ is a polytope. Fixing $f$ and $y$, we define 
$$\SOS(f,y) \ = \ \{\lambda \in C(y) \ : \ f(Q, \lambda, y) \ \text{ is a sum of squares}\}$$
where $Q = \New(f)$. We have already seen in the proof of Proposition \ref{prop:counterexample} that the structure of $\SOS(f,y)$ is unclear and highly depends on the convex combinations of $y$. It is formulated as an open question in \cite{Reznick:AGI}, whether one can say something about $\SOS(f,y)$ for fixed $f$ and $y$. 
For this, let 
$$Q \ = \ Q_1^{(i)}\cup\dots\cup Q_{r(i)}^{(i)}$$ 
be a triangulation of $Q$ for $1 \leq i \leq t$, where $t$ is the number of triangulations of $Q$ without using new vertices. We are interested in those simplices $Q^{(i)}_j$ that contain the point $y \in \Int(Q)$ and their maximal mediated sets $(Q^{(i)}_j)^*$. Recall that for every lattice simplex $\Delta$ with vertex set $\hat \Delta$ we denote $\Delta^*$ as the maximal $\hat \Delta$-mediated set (see Section \ref{SubSec:Agiforms}).

\begin{thm}
 Let $Q \subset \R^n$ be a lattice $n$-polytope, $y \in \Int(Q) \cap \N^n$ and $f \in AP^y_Q$ be an agiform. Then $\SOS(f,y) = C(y)$, i.e., \emph{every} agiform is a sum of squares, if and only if $y \in Q^{(i)}_j$ implies $y \in (Q_{j}^{(i)})^*$ for every $1\leq i\leq t$ and $1 \leq j \leq r(i)$. 
\label{Thm:ReznickProblem}
\end{thm}
 
\begin{proof}
 Assume $y \in Q^{(i)}_j \Ra y \in (Q_{j}^{(i)})^*$ for every $1\leq i\leq t$ and $1 \leq j \leq r(i)$. Let $\lambda\in C(y)$ with $f(Q,\lambda,y)$ being the corresponding agiform. By \cite[Theorem 7.1]{Reznick:AGI}, every agiform can be written as a convex combination of simplicial agiforms. In fact, following the proof in \cite[Theorem 7.1]{Reznick:AGI}, it can be verified that the vertices of the corresponding simplicial agiforms form a subset of the vertices of $Q$, since the set $C(y)$ of convex combinations of $y$ is a polytope with vertices being a subset of $\V(Q)$. Hence, these agiforms come from triangulating the polytope $Q$ into simplices without using new vertices. Since $y \in Q^{(i)}_j \Ra y \in ({Q_j^{(i)}})^*$ for every $i,j$, by Theorem \ref{thm:reznick}, the corresponding simplicial agiforms are always sums of squares and since $f(Q,\lambda,y)$ is a sum of them, the claim follows.

For the reverse direction, assume $y \in Q_{j}^{(i)}$ and $y \notin (Q_{j}^{(i)})^*$ for some $i,j$. We prove that this implies $\SOS(f,y) \neq C(y)$. Suppose $\V(Q) = \{v_1,\dots,v_m\}$. Then $C(y)$ is a polytope of dimension $d = m - (n+1)$. Let
$$f(Q,\lambda,y) \ = \ \sum_{i=1}^m \lambda_i(\mu_1,\dots,\mu_d)x^{v_i} - x^y$$
be the corresponding agiforms. Note that the coefficients $\lambda_i$ depend on $d$ parameters $\mu_1,\dots,\mu_d$, since $\dim C(y) = d$. By assumption, there exist $a_1,\dots,a_d \in \R_{> 0}$ such that the corresponding agiform $f(Q,\lambda,y)_{|(\mu_1,\dots,\mu_d) = (a_1,\dots,a_d)} = g$ is a simplicial agiform with respect to the simplex $Q_j^{(i)}$. Since $y \in Q_j^{(i)}$ but $y \notin (Q_{y,k}^{(i)})^*$, the agiform $g$ is not a sum of squares. By continuity, we can construct a sequence $(\mu_1,\dots,\mu_d)$ converging against $(a_1,\dots,a_d)$ with the properties that $f(Q,\lambda,y)_{|(\mu_1,\dots,\mu_d) = (a_1+\varepsilon,\dots,a_d+\varepsilon)}$ is an agiform for some $\varepsilon > 0$ with its support equal to $\{v_1,\dots,v_m,y\}$ and not being a sum of squares, since, otherwise, if every sequence member is a sum of squares, this will also hold for the limit agiform $g$ corresponding to $(a_1,\dots,a_d)$ since the cone of sums of squares is closed. Hence, $\SOS(f,y) \neq C(y)$.
\end{proof}

\begin{example}
 Let again 
$$Q \ = \ \conv\{v_0,v_1,v_2,v_3\} \ = \ \conv\{(0,0), (4,0), (4,2), (2,4)\}$$ 
as in the proof of Proposition \ref{prop:counterexample}. There are six interior lattice points in $Q$ given by
$$\Int(Q) \cap \N^n \ = \ \{(1,1), (2,1), (3,1), (2,2), (2,3), (3,2)\}.$$
Since $Q$ has four vertices, $C(y)$ for $y\in (\Int(Q) \cap \N^n)$ has a free parameter $\lambda_3$ (see proof of Proposition \ref{prop:counterexample}). In the following table, for all $y\in (\Int(Q) \cap \N^n)$, we provide the range of the free parameter $\lambda_3$ yielding valid convex combinations for $y$ as well as the set $\SOS(f,y)$.

\renewcommand{\arraystretch}{1.5}
$$
\begin{tabular}{c|c|c}\hline 
 $y$ & $\lambda_3$ & $\SOS(f,y)$ \\ \hline 
 $(1,1)$ & $\frac{1}{6}\leq \lambda_3\leq \frac{1}{4}$ & $\lambda_3\in [0.191;\frac{1}{4}]$ \\ \hline 
 $(2,1)$ & $0\leq \lambda_3\leq \frac{1}{4}$ & $\lambda_3\in [0;\frac{1}{4}]$ \\ \hline 
 $(3,1)$ & $0\leq \lambda_3\leq \frac{1}{4}$ & $\lambda_3\in [0;\frac{1}{4}]$ \\ \hline 
 $(2,2)$ & $\frac{1}{3}\leq \lambda_3\leq \frac{1}{2}$ & $\lambda_3\in \{\frac{1}{2}\}$ \\ \hline 
 $(2,3)$ & $\frac{2}{3}\leq \lambda_3\leq \frac{3}{4}$ & $\lambda_3\in [0.683;\frac{3}{4}]$ \\ \hline 
 $(3,2)$ & $\frac{1}{6}\leq \lambda_3\leq \frac{1}{2}$ & $\lambda_3\in [\frac{1}{4};\frac{1}{2}]$ \\ \hline 
\end{tabular}
$$

The sets $\SOS(f,y)$ are computed with \emph{SOSTOOLS}, see \cite{Parrilo:SOSTOOLS}. Note that $Q$ has two different triangulations in this case (see Figure \ref{Fig:Triangulations}). The lattice points $(2,1)$ and $(3,1)$ are the only lattice points that satisfy $y \in Q_j^{(i)} \Ra y \in (Q_j^{(i)})^*$ for all $i \in \{1,2\}$ and $j \in \{1,\ldots,r(i)\}$. Hence, exactly for $y \in \{(2,1), (3,1)\}$, every agiform is a sum of squares.
\end{example}

\begin{center}
\begin{figure}
\ifpictures
$$
\includegraphics[width=0.25\linewidth]{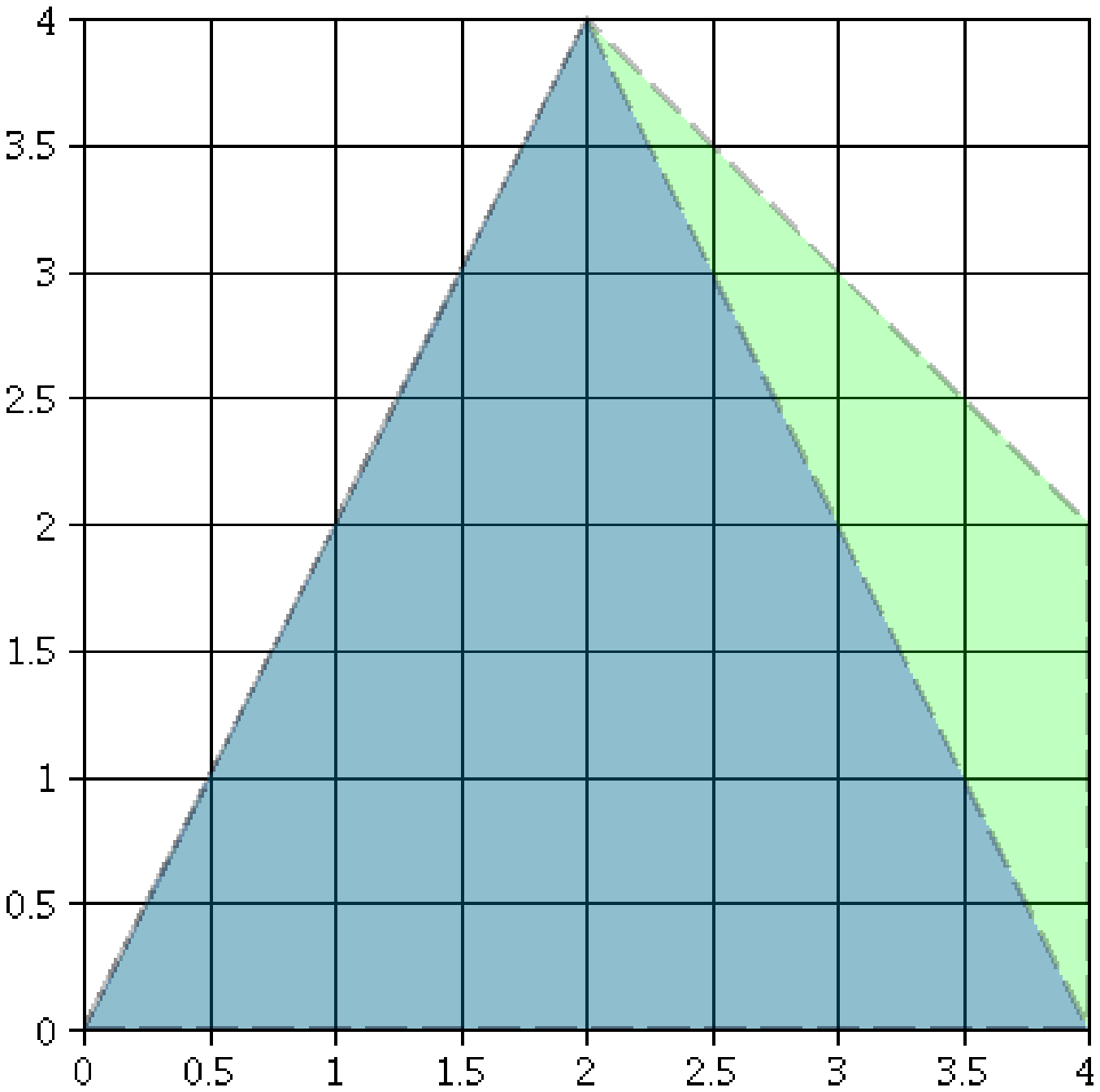} \qquad \qquad
\includegraphics[width=0.25\linewidth]{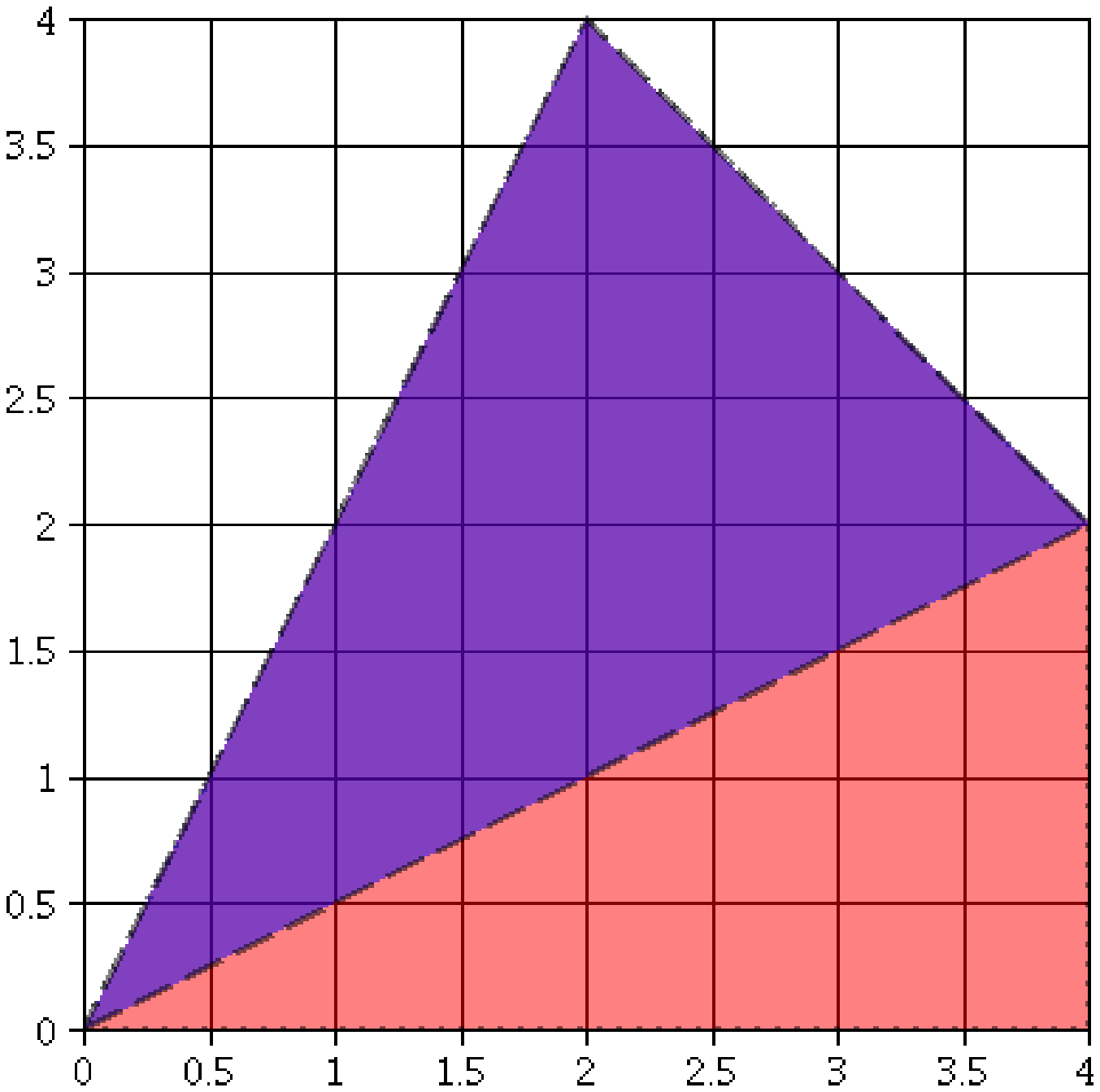}
$$
\fi
\caption{The two triangulations of $Q$.}
\label{Fig:Triangulations}
\end{figure}
\end{center}

\section{Outlook}
\label{Sec:outlook}
We want to give an outlook for possible future research. Starting with the section $\Sigma_{n,2d}^y$, we renew some open questions already stated in \cite{Reznick:AGI}. Is there an algorithm to compute $\Delta^*$ that is more efficient as the one in \cite{Reznick:AGI}? What can be said about the asymptotic behavior of $\Delta^*$, in particular, what is the, say, ``probability'' that a simplex is an $H$-simplex? This is settled for $\mathbb R^2$ in Corollary \ref{cor:R2}, but seems to be completely open for $n > 2$. Considering this problem from the viewpoint of toric geometry (see Theorem \ref{Thm:toric}), it would be a breakthrough to characterize simplices that are normal and their corresponding toric ideals being generated by quadrics. In Section \ref{Sec:SOPC}, we introduced the convex cone $C_{n,2d}$ of sums of nonnegative circuit polynomials, which serve as nonnegativity certificates different than sums of squares. From a practical viewpoint, the major problem is to determine the complexity of checking membership in $C_{n,2d}$. In particular, when is every nonnegative polynomial a sum of nonnegative circuit polynomials? As already mentioned in Section 7, the case of polynomials with simplex Newton polytopes is solved in \cite{Iliman:deWolff:GP} via geometric programming generalizing earlier work by Ghasemi in Marshall \cite{Ghasemi:Marshall:GPGlobal,Ghasemi:Marshall:GPSemialgebraic}.

From the viewpoint of amoeba theory one evident conjecture is that Theorem \ref{Thm:AmoebaSolidness} can be generalized to arbitrary \textit{complex} polynomials supported on a circuit. Taking into account the corresponding literature, in particular \cite{Passare:Rullgard:Spine,TTdW:genusone}, an answer to this conjecture can be considered as the final piece missing in order to completely characterize amoebas supported on a circuit. 

In our opinion, the most interesting question is whether similar approaches can be generalized to more general (sparse) polynomials and, in accordance, how much deeper the observed connection between the a priori very distinct mathematical topics ``amoebas'' and ``nonnegativity of real polynomials'' is? We believe that exploiting methods from amoeba theory might eventually yield fundamental progress in understanding nonnegativity of real polynomials.

\bibliographystyle{amsplain}
\bibliography{Amoeba_Nonnegative_ArXiv}

\providecommand{\bysame}{\leavevmode\hbox to3em{\hrulefill}\thinspace}
\providecommand{\MR}{\relax\ifhmode\unskip\space\fi MR }
\providecommand{\MRhref}[2]{%
  \href{http://www.ams.org/mathscinet-getitem?mr=#1}{#2}
}
\providecommand{\href}[2]{#2}
\begin{thebibliography}{10}

\bibitem{Ahmadi:et:al:convex}
A.A. Ahmadi, A.~Olshevsky, P.A. Parrilo, and J.N. Tsitsiklis,
  \emph{N{P}-hardness of deciding convexity of quartic polynomials and related
  problems}, Math. Program. \textbf{137} (2013), no.~1-2, Ser. A, 453--476.

\bibitem{Bjoerner:et:al:Matroids}
A.~Bj{\"o}rner, M.~Las Vergnas, B.~Sturmfels, N.~White, and G.M. Ziegler,
  \emph{Oriented matroids}, second ed., Encyclopedia of Mathematics and its
  Applications, vol.~46, Cambridge University Press, Cambridge, 1999.

\bibitem{Blekherman:Parrilo:Thomas}
G.~Blekherman, P.A. Parrilo, and R.R. Thomas, \emph{Semidefinite optimization
  and convex algebraic geometry}, MOS-SIAM Series on Optimization, vol.~13,
  SIAM and the Mathematical Optimization Society, Philadelphia, 2013.

\bibitem{Brieskorn:Knoerrer}
E.~Brieskorn and H.~Kn{\"o}rrer, \emph{Plane algebraic curves}, Modern
  Birkh\"auser Classics, Birkh\"auser/Springer Basel AG, Basel, 1986.

\bibitem{Bruns:et:al}
W.~Bruns, J.~Gubeladze, and N.V. Trung, \emph{Normal polytopes, triangulations,
  and {K}oszul algebras}, J. Reine Angew. Math. \textbf{485} (1997), 123--160.

\bibitem{Reznick:realzeros}
M.D. Choi, Y.T. Lam, and B.~Reznick, \emph{Real zeros of positive semidefinite
  forms. {I}}, Math. Z. \textbf{171} (1980), no.~1, 1--26.

\bibitem{deWolff:Diss}
T.~de~Wolff, \emph{On the {G}eometry, {T}opology and {A}pproximation of
  {A}moebas}, Ph.D. thesis, Goethe {U}niversity, Frankfurt am Main, 2013.

\bibitem{Einsiedler:et:al}
M.~Einsiedler, D.~Lind, R.~Miles, and T.~Ward, \emph{Expansive subdynamics for
  algebraic {$\mathbb{Z}^d$}-actions}, Ergodic {T}heory {D}ynam. {S}ystems
  \textbf{21} (2001), no.~6, 1695--1729.

\bibitem{Fidalgo:Kovacec}
C.~Fidalgo and A.~Kovacec, \emph{Positive semidefinite diagonal minus tail
  forms are sums of squares}, Math. Z. \textbf{269} (2011), no.~3-4, 629--645.

\bibitem{Forsberg:Passare:Tsikh}
M.~Forsberg, M.~Passare, and A.~Tsikh, \emph{Laurent determinants and
  arrangements of hyperplane amoebas}, Adv. {M}ath. \textbf{151} (2000),
  45--70.

\bibitem{GKZ:discriminant}
I.M. Gelfand, M.M. Kapranov, and A.V. Zelevinsky, \emph{Discriminants,
  resultants and multidimensional determinants}, Modern Birkh\"auser Classics,
  Birkh\"auser Boston Inc., Boston, MA, 2008.

\bibitem{Ghasemi:Marshall:GPGlobal}
M.~Ghasemi and M.~Marshall, \emph{Lower bounds for polynomials using geometric
  programming}, SIAM J. Optim. \textbf{22} (2012), no.~2, 460--473.

\bibitem{Ghasemi:Marshall:GPSemialgebraic}
\bysame, \emph{Lower bounds for a polynomial on a basic closed semialgebraic
  set using geometric programming},  (2013), Preprint, {\sf arxiv:1311.3726}.

\bibitem{Gruenbaum}
B.~Gr{\"u}nbaum, \emph{Convex polytopes}, second ed., Graduate Texts in
  Mathematics, vol. 221, Springer-Verlag, New York, 2003.

\bibitem{Gubeladze:normality}
J.~Gubeladze, \emph{Convex normality of rational polytopes with long edges},
  Adv. Math. \textbf{230} (2012), no.~1, 372--389.

\bibitem{Hilbert:Seminal}
D.~Hilbert, \emph{Ueber die {D}arstellung definiter {F}ormen als {S}umme von
  {F}ormenquadraten}, Math. Ann. \textbf{32} (1888), no.~3, 342--350.

\bibitem{Iliman:deWolff:GP}
S.~Iliman and T.~de~Wolff, \emph{Lower bounds for polynomials with simplex
  newton polytopes based on geometric programming}, 2014, Preprint, {\sf
  arXiv:1402.6185}.

\bibitem{Kenyon:Okounkov:Sheffield}
R.~Kenyon, A.~Okounkov, and S.~Sheffield, \emph{Dimers and amoebae}, Ann. of
  Math. (2) \textbf{163} (2006), no.~3, 1019--1056.

\bibitem{Koelman:toric}
R.J. Koelman, \emph{A criterion for the ideal of a projectively embedded toric
  surface to be generated by quadrics}, Beitr\"age Algebra Geom. \textbf{34}
  (1993), no.~1, 57--62.

\bibitem{Lasserre:sparse}
J.B. Lasserre, \emph{Convergent {SDP}-relaxations in polynomial optimization
  with sparsity}, SIAM J. Optim. \textbf{17} (2006), no.~3, 822--843.

\bibitem{Lasserre:Buch}
\bysame, \emph{Moments, positive polynomials and their applications}, Imperial
  College Press Optimization Series, vol.~1, Imperial College Press, London,
  2010.

\bibitem{Laurent:Survey}
M.~Laurent, \emph{Sums of squares, moment matrices and optimization over
  polynomials}, Emerging applications of algebraic geometry, IMA Vol. Math.
  Appl., vol. 149, Springer, New York, 2009, pp.~157--270.

\bibitem{MacLagan:Sturmfels}
D.~Maclagan and B.~Sturmfels, \emph{Introduction to {T}ropical {G}eometry},
  Amer.\ Math.\ Soc., Providence, R.I., 2015.

\bibitem{Mikhalkin:Annals}
G.~Mikhalkin, \emph{Real algebraic curves, the moment map and amoebas}, Ann.\
  Math. \textbf{151} (2000), 309--326.

\bibitem{Mikhalkin:Survey}
\bysame, \emph{Amoebas of algebraic varieties and tropical geometry}, Different
  faces of geometry (S.~K. Donaldson, Y.~Eliashberg, and M.~Gromov, eds.),
  Kluwer, 2004, pp.~257--300.

\bibitem{Nie:Discriminants}
J.~Nie, \emph{Discriminants and nonnegative polynomials}, J. Symbolic Comput.
  \textbf{47} (2012), no.~2, 167--191.

\bibitem{Passare:Rullgard:Spine}
M.~Passare and H.~Rullg{\aa}rd, \emph{Amoebas, {M}onge-{A}mp\'{e}re measures
  and triangulations of the {N}ewton polytope}, Duke Math.\ J. \textbf{121}
  (2004), no.~3, 481--507.

\bibitem{Passare:Tsikh:Survey}
M.~Passare and A.~Tsikh, \emph{Amoebas: their spines and their contours},
  Idempotent mathematics and mathematical physics, Contemp. Math., vol. 377,
  Amer. Math. Soc., pp.~275--288.

\bibitem{Parrilo:SOSTOOLS}
S.~Prajna, A.~Papachristodoulou, P.~Seiler, and P.A. Parrilo, \emph{S{OSTOOLS}
  and its control applications}, Positive polynomials in control, Lecture Notes
  in Control and Inform. Sci., vol. 312, Springer, Berlin, 2005, pp.~273--292.

\bibitem{Purbhoo}
K.~Purbhoo, \emph{A {N}ullstellensatz for amoebas}, Duke {M}ath. {J}.
  \textbf{14} (2008), no.~3, 407--445.

\bibitem{Reznick:AGI}
B.~Reznick, \emph{Forms derived from the arithmetic-geometric inequality},
  Math. Ann. \textbf{283} (1989), no.~3, 431--464.

\bibitem{Reznick:Blenders}
\bysame, \emph{{Blenders.}}, {Notions of positivity and the geometry of
  polynomials. Dedicated to the memory of Julius Borcea}, Basel: Birkh\"auser,
  2011, pp.~345--373.

\bibitem{Rullgard:Diss}
H.~Rullg{\aa}rd, \emph{Topics in geometry, analysis and inverse problems},
  Ph.D. thesis, Stockholm {U}niversity, 2003.

\bibitem{Schroeter:deWolff:Boundary}
F.~Schroeter and T.~de~Wolff, \emph{The boundary of amoebas}, 2013, Preprint,
  {\sf arXiv:1310.7363}.

\bibitem{Sottile:Book:RealSolutions}
F.~Sottile, \emph{Real solutions to equations from geometry}, University
  Lecture Series, vol.~57, American Mathematical Society, Providence, RI, 2011.

\bibitem{Sturmfels:toric}
B.~Sturmfels, \emph{Equations defining toric varieties}, Algebraic geometry --
  {S}anta {C}ruz 1995, Proc. Sympos. Pure Math., vol.~62, Amer. Math. Soc.,
  Providence, RI, 1997, pp.~437--449.

\bibitem{TTdW:genusone}
T.~Theobald and T.~de~Wolff, \emph{Amoebas of genus at most one}, Adv. Math.
  \textbf{239} (2013), 190--213.

\bibitem{Theobald:deWolff:Trinomials}
T.~Theobald and T.~de~Wolff, \emph{Norms of roots of trinomials}, 2014, To
  appear in Math. Ann., see also {\sf arXiv:1411.6552}.

\end{thebibliography}

\end{document}